\documentclass[12pt]{article}

\usepackage{amssymb,amsmath,amsfonts,amssymb}
\usepackage{graphics,graphicx,color,hhline}
\usepackage{bbm} 
\usepackage{euscript}  
\usepackage{authblk} 
 \usepackage{mathtools}
\usepackage{cases}
\def\@abssec#1{\vspace{.05in}\footnotesize \parindent .2in 
{\bf #1. }\ignorespaces} 

\graphicspath{{/EPSF/}{Figures/}}   

\newtheorem{theorem}{Theorem}[section]
\newtheorem{lemma}[theorem]{Lemma}
\newtheorem{proposition}[theorem]{Proposition}
\newtheorem{corollary}[theorem]{Corollary}

\def \Rm {\mathbb R}
\def \Nm {\mathbb N}

\def\II{\mathbb{I}}
\newcommand{\eps}{\varepsilon}

\newcommand{\ds}{\displaystyle}

\newcommand{\calH}{\mathcal H}
\newcommand{\calL}{\mathcal L}

\newcommand{\calF}{\mathcal F}

\newcommand{\calE}{\mathcal E}
\newcommand{\Tr}{\textnormal{Tr}}

\newcommand{\calJ}{\mathcal J}

\newcommand{\calT}{\mathcal T}

\def\fref#1{{\rm (\ref{#1})}}

\newcommand{\cout}[1]{}

\def\un{{\mathbbmss{1}}}

\newcommand{\be}{\begin{equation}}
\newcommand{\ee}{\end{equation}}
\newcommand{\bea}{\begin{eqnarray}}
\newcommand{\eea}{\end{eqnarray}}
\newcommand{\bee}{\begin{eqnarray*}}
\newcommand{\eee}{\end{eqnarray*}}
\newcommand{\bal}{\begin{align*}}
\newcommand{\eal}{\end{align*}}

\def\Hp{{H}^1_{per}}
\def\Hmp{{H}^{-1}_{per}}

\begin{document}
{\title{The quantum drift-diffusion model: existence and exponential convergence to the equilibrium}}

 \author{Olivier Pinaud \footnote{pinaud@math.colostate.edu}}
 \affil{Department of Mathematics, Colorado State University, Fort Collins CO, 80523}

\maketitle

\begin{abstract}
  This work is devoted to the analysis of the quantum drift-diffusion model derived by Degond et al in \cite{QET}. The model is obtained as the diffusive limit of the quantum Liouville-BGK equation, where the collision term is defined after a local quantum statistical equilibrium. The corner stone of the model is the closure relation between the density and the current, which is nonlinear and nonlocal, and is the main source of the mathematical difficulties. The question of the existence of solutions has been open since the derivation of the model, and we provide here a first result in a one-dimensional periodic setting. The proof is based on an approximation argument, and exploits some properties of the minimizers of an appropriate quantum free energy. We investigate as well the long time behavior, and show that the solutions converge exponentially fast to the equilibrium. This is done by deriving a non-commutative logarithmic Sobolev inequality for the local quantum statistical equilibrium.
\end{abstract}

\section{Introduction}

The quantum drift-diffusion model was derived in \cite{QET} by Degond et al, with the goal of describing the diffusive behavior of quantum particles. The widely used classical drift-diffusion model \cite{MarkoRing} is indeed not accurate as the size of electronic devices decreases, and models accounting for quantum effects are necessary. The quantum drift-diffusion model is obtained as the (informal) diffusive limit of the quantum Liouville-BGK equation 
\be \label{liouville}
i \hbar \partial_t \varrho=[H, \varrho]+i \hbar Q(\varrho),
\ee
where $\varrho$ is the density operator, i.e. a self-adjoint nonnegative trace class operator that models a statistical ensemble of particles (here electrons), $H$ is a given Hamiltonian, $[\cdot,\cdot]$ denotes the commutator between two operators, and $Q$ is a collision operator. The original feature of \fref{liouville} lies in the definition of $Q$, which is of BGK type \cite{BGK}, and takes the form, in its simplest version, 
\be \label{col}
Q(\varrho)=\frac{1}{\tau} \left(\varrho_e(\varrho) -\varrho\right),
\ee 
where $\tau$ is a relaxation time and $\varrho_e(\varrho)$ is a so-called \textit{quantum statistical equilibrium}. The main motivation behind equations \fref{liouville}-\fref{col} is to describe the collective dynamics of many-particles quantum systems, and in particular to derive reduced quantum fluid models. To this end, Degond and Ringhofer translates in \cite{DR} Levermore's entropy closure strategy \cite{levermore} to the quantum case. As in the kinetic situation, this requires the introduction of some statistical equilibria, which, in the quantum case, are  minimizers of the free energy
$$
F(\varrho)= T \Tr (\beta(\varrho)) + \Tr (H \varrho),
$$ 
where $T$ is the temperature (we will set $\hbar=T=1$ for simplicity, as well as all physical constants), $\beta$ is an entropy function, and $\Tr$ denotes operator trace. The free energy $F$ is minimized under a given set of constraints on the moments of $\varrho$, which include for instance the density, the momentum, or the energy, and these constraints present the particularity of being \textit{local}. In other terms, when prescribing the first moment only for simplicity of the exposition, $F$ is minimized under the constraint that the local density $n[\varrho](x)$ of $\varrho$ is equal to a given function $n(x)$. If $\varrho$ is associated to an integral kernel $\rho(x,y)$, then $n[\varrho](x)$ is simply formally $\rho(x,x)$. The analysis of the minimization problem alone is not trivial, mostly because of the local character of the constraints, and was addressed in \cite{MP-JSP,MP-KRM} in the cases where the first two moments of $\varrho$ are prescribed. The case of higher order moments is still open.

In its simplest form, the collision operator $Q$ is then defined after the equilibrium $\varrho_e(\varrho)$, where $\varrho_e(\varrho)$ is a minimizer of the free energy under the constraint $n[\varrho_e](x)=n[\varrho](x)$. When $\beta$ is the Boltzmann entropy, then $\varrho_e$ is referred to as the \textit{quantum Maxwellian}. Note that a rigorous construction of the latter as a minimizer of the constrained $F$ is not direct, see the discussions of this fact in \cite{MP-JDE}. With the so-defined $\varrho_e(\varrho)$ at hand, one can then consider the evolution problem \fref{liouville}. The main difficulty in the analysis is the fact that the map $\varrho \mapsto \varrho_e(\varrho)$ is nonlinear, and foremost that it is defined via an implicit intricate nonlocal relation (see further equation \fref{clos}). The existence of solutions to \fref{liouville} was proved in \cite{MP-JDE} in a one-dimensional setting, the uniqueness and higher dimensional settings remain open problems.  

The Quantum Drift-Diffusion model (QDD in the sequel) is then obtained as the diffusive limit of \fref{liouville} when $\beta$ is the Boltzmann entropy. For $\eps=\tau/\overline{t}$, where $\overline{t} \gg \tau$ is some characteristic time, it is shown formally in \cite{QET}, that a solution $\varrho_\eps(t)$ to an appropriately rescaled version of \fref{liouville} converges as $\eps \to 0$ to a quantum Maxwellian of the form $\exp(-(H+A(t,x)))$ (defined in the functional calculus sense), where $A(t,x)$ is the so-called \textit{quantum chemical potential} and satisfies the system, that will be complemented with boundary conditions further, 



\begin{numcases}{} 
\ds \frac{\partial n}{\partial t}+ \nabla \cdot \big(n  \nabla (A-V) \big)=0, \label{qdd}\\[3mm]
\ds -\Delta V=n, \label{poisson}\\[3mm]
\ds  n=n[e^{-(H+A)}]=\sum_{p \in \Nm} e^{-\lambda_p} |\phi_p|^2 \label{clos}.
\end{numcases}
Above, $V$ is the Poisson potential that accounts for the electrostatic interactions between the electrons. The corner stone of the above system is the nonlinear nonlocal closure relation \fref{clos}, that expresses the relationship between the density $n$ and the potential $A$: $n$ is the local density of the operator $\exp(-(H+A))$. Assuming the Hamiltonian $H+A$ has a compact resolvant, the second inequality in \fref{clos} holds for $(\lambda_p,\phi_p)_{p \in \Nm}$ the spectral decomposition of $H+A$. Since $A$ is the main quantity here, the system \fref{qdd}-\fref{poisson}-\fref{clos} is probably best seen as an evolutionary problem on $A$ rather than on $n$.

One of our objectives in this work is to construct solutions to \fref{qdd}-\fref{poisson}-\fref{clos}. The question has been open since the derivation of the model in \cite{QET}. Some progress was made in \cite{QDD-SIAM}, where solutions to a semi-discretized (w.r.t. the time variable) system were constructed as minimizers of an appropriate functional. The continuum limit was not performed in \cite{QDD-SIAM}, mostly for two reasons: (i) uniform estimates in the discretization parameter were missing; they require some lower bounds on the density $n$ that were not available at that time, and (ii) the closure relation \fref{clos} was not yet well understood mathematically. We provide here the missing ingredients needed to pass to the limit, and therefore obtain the first result of existence of solutions for \fref{qdd}-\fref{poisson}-\fref{clos}: we derive a lower bound on the density assuming the initial state is sufficiently close to the equilibrium, and based on our previous analysis of the minimization problem in \cite{MP-JSP,MP-KRM,MP-JDE}, we have now the technical tools to obtain \fref{clos} as the limit of the discretized version. We will work in a one-dimensional setting with periodic boundary conditions. The latter can directly be replaced by Neumann boundary conditions, while Dirichlet boundary conditions would create additional technical difficulties since the density would vanish at the boundary. The limitation to one-dimensional domains is addressed further in the paper, it pertains to the derivation of the aforementioned lower bound that involves a Sobolev embedding. Note that the system \fref{qdd}-\fref{poisson}-\fref{clos} can be written as a gradient flow in the Wasserstein space, but because of the complexity of the relation between $n$ and $A$, we were not able to use the standard theory.

Our other objective is to investigate the long time limit of \fref{qdd}-\fref{poisson}-\fref{clos}, and in particular to obtain an exponential convergence to the equilibrium. This will be achieved by deriving some non-commutative logarithmic Sobolev inequality satisfied by the operator $\exp(-(H+A(t,x)))$, in the spirit of those of \cite{carbone}.

As a conclusion of this introduction, we would like to point out that a different model is also referred to as the quantum drift-diffusion model in the literature. This model, sometimes also called  the ``density gradient model'', is a classical drift-diffusion model corrected by a quantum term. As was shown in \cite{QET}, it is actually obtained in the semi-classical limit of the quantum drift-diffusion model considered here, by accounting for the first-order correction. In the density-gradient model, the closure relation is local and much simpler than \fref{clos}, and $A$ is related to the so-called Bohm potential $\Delta \sqrt{n} / \sqrt{n}$, leading to a fourth-order parabolic equation of the form
\be \label{p4}
\frac{\partial n}{\partial t}+ \nabla \cdot \left( n \nabla \left( \frac{\Delta \sqrt{n}}{\sqrt{n}} -\log(n) \right) \right)=0.
\ee 
One disadvantage of this model is the introduction of high order derivatives, that do not appear in \fref{qdd}-\fref{poisson}-\fref{clos}. A closely related model, obtained in the zero temperature limit (the term $\log(n)$ then vanishes in \fref{p4}), is the Derrida-Lebowitz-Speer-Spohn equation \cite{DLSS1,DLSS2}, that was extensively studied mathematically in the recent years. The existence and uniqueness of solutions was first limited to one-dimensional domains for the same technical reason as here, see \cite{bleher, gualdani,JungPin}. The existence of solutions was then extended to multi-dimensional domains in \cite{gianazza} using optimal transport techniques, and in \cite{jungmatthes} with more direct methods.

Note that the QDD system \fref{qdd}-\fref{poisson}-\fref{clos} inherits some of the technical difficulties of \fref{p4} (or vice-versa), in particular the strict positivity of the density, and presents new challenges as the closure relation is not local. In particular, the monoticity property of the high-order non-linear term in \fref{p4} obtained in \cite{JungPin}, which is the main ingredient for proving uniqueness, does not seem to generalize to our case and we are limited to an existence result.

The paper is structured as follows: in Section \ref{prelimi}, we introduce some notation and important results about the minimization of the free energy $F(\varrho)$; our main theorem is stated in Section \ref{main}, and its proof is given in Section \ref{proofth}. A technical lemma is finally proved in the Appendix.

{\bf Acknowledgment.} This work was supported by NSF CAREER grant DMS-1452349. 

\section{Preliminaries} \label{prelimi}

We start by introducing some notation. 

\paragraph{Notation.} Our domain $\Omega$ is the 1-torus $[0,1]$. We will denote by $L^r(\Omega)$ , $r\in [1,\infty]$, the usual Lebesgue spaces of complex-valued functions, and by $W^{k,r}(\Omega)$, the standard Sobolev spaces. We introduce as well $H^k=W^{k,2}$, and $(\cdot,\cdot)$ for the Hermitian product on $L^2(\Omega)$ with the convention $(f,g)=\int_\Omega \overline{f} g dx$. We will use the notations $\nabla=d/dx$ and $\Delta=d^2/dx^2$ for brevity. For a given exterior potential $V_0 \in L^\infty(\Omega)$, we consider then the Hamiltonian
\be H=-\Delta + V_0 \; \textrm{with domain} \; 
D(H)=\left\{u\in H^2(\Omega):\,u(0)=u(1),\,\nabla u(0)=\nabla u(1)\right\}.\label{domainH} \ee
The free Hamiltonian $-\Delta$ is denoted by $H_0$, and $\Hp$ is the space of $H^1(\Omega)$ functions $u$ that satisfy $u(0)=u(1)$. Its dual space is $\Hmp$. We shall denote by $\calL(L^2(\Omega))$ the space of bounded operators on $L^2(\Omega)$, by $\calJ_1 \equiv \calJ_1(L^2)$ the space of trace class operators on $L^2(\Omega)$,  and more generally by $\calJ_r$ the Schatten space of order $r$. 

A density operator is defined as a nonnegative trace class, self-adjoint operator on $L^2(\Omega)$. For $|\varrho|=\sqrt{\varrho^* \varrho}$, we introduce the following space:
$$\calE=\left\{\varrho\in \calJ_1,\mbox{ such that } \overline{\sqrt{H}|\varrho|\sqrt{H}}\in \calJ_1\right\},$$
where $\overline{\sqrt{H}|\varrho|\sqrt{H}}$ denotes the extension of the operator $\sqrt{H}|\varrho|\sqrt{H}$ to $L^2(\Omega)$. We will drop the extension sign in the sequel to ease notation. The space $\calE$ is a Banach space when endowed with the norm
$$\|\varrho\|_{\calE}=\Tr \big(|\varrho| \big)+\Tr\big(\sqrt{H}|\varrho|\sqrt{H}\big),$$
where $\Tr$ denotes operator trace. The energy space is the following closed convex subspace of $\calE$:
$$\calE_+=\left\{\varrho\in \calE:\, \varrho\geq 0\right\}.$$
Note that operators in $\calE_+$ are automatically self-adjoint since they are bounded and positive on the complex Hilbert space $L^2(\Omega)$. For any $\varrho\in \calJ_1$ with $\varrho=\varrho^*$, one can associate a real-valued local density $n[\varrho](x)$, formally defined by $n[\varrho](x)=\rho(x,x),$ where $\rho$ is the integral kernel of $\varrho$. The density $n[\varrho]$ can be in fact identified uniquely by the following weak formulation:
$$\forall \phi\in L^\infty(\Omega),\quad \Tr \big(\Phi \varrho \big)=\int_\Omega \phi(x)n[\varrho](x)dx,
$$
where, in the left-hand side, $\Phi$ denotes the multiplication operator by $\phi$ and belongs to $\calL(L^2(\Omega))$. In the sequel, we will consistently identify a function and its associated multiplication operator. Throughout the paper, $C$ will denote a generic constant that might differ from line to line.

The next step is to introduce the minimization problem that is at the core of the closure relation \fref{clos}.

\paragraph{The minimization problem.} We will work with the Boltzmann entropy $\beta(x)=x\log x -x$. For $\varrho \in \calE_+$ and $V[n[\varrho]]\equiv V$ the Poisson potential satisfying \fref{poisson} with boundary conditions $V(0)=V(1)=0$ and density $n[\varrho]$ on the right-hand side, we introduce the free energy $F$ defined by 
\be \label{deffreeE}
F(\varrho)=\Tr \big( \beta(\varrho) \big)+ \Tr \big( \sqrt{H_0} \varrho \sqrt{H_0} \big) + \Tr \big( V_0  \varrho \big)+\frac{1}{2} \|\nabla V\|^2 _{L^2}.
\ee
Note that all terms above are well defined when $\varrho \in \calE_+$: on the one hand, it is direct to see that $n[\varrho] \in W^{1,1}(\Omega)$, and therefore elliptic regularity shows that the last term above is finite; on the other hand, the entropy term is finite according to \fref{ent2} in Lemma \ref{below} further. It is moreover a classical fact that the mapping $\varrho \mapsto \Tr(\beta(\varrho))$ is strictly convex (see e.g. \cite{MP-JSP}, Lemma 3.3, for a proof), and therefore $F$ is strictly convex as well.

The theorem below characterizes  the minimizers of $F$ under a global density constraint. They will be shown to be the equilibrium solutions to \fref{qdd}-\fref{poisson}-\fref{clos}. The proof can be found in \cite{NierCPDE}, up to minor modifications. 

\begin{theorem} \label{SP} (The global minimization problem). Let $N \in \Rm_+^*$. The problem
$$
\textrm{min } F(\varrho), \quad \textrm{for} \quad \varrho \in \calE_+ \quad \textrm{with} \quad \Tr \big( \varrho \big)= N,
$$
admits a unique solution that reads
$$
\varrho_\infty=\exp \big(- (H + A_\infty)\big),
$$
where $A_\infty=V_\infty-\epsilon_F \in \Hp$, for $\epsilon_F$ a constant and 
$$
- \Delta V_\infty= n[\varrho_\infty], \qquad V_\infty(0)=V_\infty(1)=0. 
$$
Moreover, there exists a constant $\underline{n}_\infty>0$ such that $n[\varrho_\infty](x)\geq \underline{n}_\infty$, $\forall x \in \overline{\Omega}$.
\end{theorem}

The minimization problem of the last theorem can be recast into a Schr\"odinger-Poisson system as in \cite{NierCPDE}: since $V_0+A_\infty \in L^\infty(\Omega)$, the operator $H+A_\infty$ with domain $D(H)$ given in \fref{domainH} is bounded below and has a compact resolvant; denoting by $(\lambda_p, \phi_p)_{p \in \Nm}$ the spectral decomposition of $H+A_\infty$ (counting multiplicity and $(\lambda_p)_{p \in \Nm}$ nondecreasing), we have, a.e. in $\Omega$,
$$
(H+V_\infty-\epsilon_F) \phi_p= \lambda_p \phi_p, \qquad \textrm{and} \qquad n[\varrho_\infty]=\sum_{p \in \Nm} e^{- \lambda_p} |\phi_p|^2.
$$
Note that $n[\varrho_\infty]$ is in $L^1(\Omega)$ since $\varrho_\infty$ is trace class. The strict positivity of the density is not addressed in \cite{NierCPDE}: it follows from the fact that the ground state $\phi_0 \in D(H) \subset C^0(\overline{\Omega})$ verifies $\phi_0(x)>0$ on $\overline{\Omega}$ according to the Krein-Rutman theorem.

The next theorem addresses the minimizers of $F$ under local constraints, which is a much more difficult problem. Its proof can be found in \cite{MP-JSP}, while the representation formula \fref{repform} is in \cite{MP-JDE} (with a slight adaptation to non-zero external potentials). Note that since the density $n$ is given, the Poisson potential is known. 

\begin{theorem} \label{local}  (The local minimization problem). Let $n \in \Hp$, nonnegative. Then, the problem
$$
\textrm{min } F(\varrho), \quad \textrm{for} \quad \varrho \in \calE_+ \quad \textrm{with} \quad n[\varrho]=n,
$$
admits a unique solution. If moreover $n>0$ on $\overline{\Omega}$, the minimizer is characterized by
$$
\varrho[n]=\exp \big(- (H + A[n])\big),
$$
where $A[n]$ belongs to  $\Hmp$ and is given by the implicit relation,  for $\varrho \equiv \varrho[n]$,
\be \label{repform}
A[n]=-V_0+\frac{1}{n}\left(\frac{1}{2}\Delta n +n[\nabla \varrho \nabla]-n[\varrho \log \varrho] \right).
\ee
\end{theorem}

\bigskip
The definition of $\varrho[n]$ above shows that the closure relation \fref{clos} is equivalent to define $A$ as the chemical potential arising from the minimization of the free energy $F(\varrho)$ under the local minimization constraint $n[\varrho]=n$. Note moreover that we have the relations
$$
n[\nabla \varrho \nabla]=-\sum_{p\in\Nm} \rho_p |\nabla \phi_p|^2, \qquad n[\varrho \log \varrho]=\sum_{p\in\Nm} \left(\rho_p \log \rho_p \right) |\phi_p|^2,
$$
which are both defined in $L^1(\Omega)$ since $\varrho[n] \in \calE_+ $. This is clear for the first term, for the second one this is a consequence of Lemma \ref{below} that shows that $\varrho \log \varrho$ is trace class. 

With Theorem \ref{local} at hand, it is possible to recast QDD as a gradient flow, at least formally. For $n$ given as in the theorem, and for $\lambda\equiv \lambda(x)$, define indeed the Lagrangian $L_n(\varrho,\lambda)$ by
$$
L_n(\varrho,\lambda)=F(\varrho)+(n[\varrho]-n, \lambda).
$$
The minimizer $\varrho[n]$ is then such that
$$
F(\varrho[n])=\min_{\varrho,\lambda} L_n(\varrho,\lambda).
$$
For $\lambda[n]$ the solution Lagrange parameter, a standard calcul of variations argument shows that
$$
\forall \delta n, \qquad \left. \frac{d}{ dt}  F(\varrho[n+ t \delta n]) \right|_{t=0}=-( \delta n, \lambda[n]).
$$
This shows, introducing  $A[n]:=V[n]+\lambda[n]$, that the $L^2$ G\^ateaux derivative of $F(\varrho[n])$ with respect to $n$, denoted $\delta F(\varrho[n])/\delta n$, verifies
$$
 \frac{\delta F(\varrho[n])}{\delta n}=-(A[n]-V[n]).
$$
The quantum drift-diffusion equation then becomes 
\be \label{GF}
 \frac{\partial n}{\partial t}- \nabla \cdot \left( n \nabla \frac{\delta F(\varrho[n])}{\delta n} \right)=0,
\ee
which is the classical form of a gradient flow in the Wasserstein space. The theory of gradient flows in Wasserstein spaces is based on the so-called \textit{geodesic $\lambda-$convexity} of the functional $F$. Once this property is established, the standard theory then provides the existence and uniqueness of solutions to equations of the form \fref{GF}, see e.g. \cite{ambrosio}. The theory covers cases where $F$ is a nonlinear, local, functional of $n$, or non-local functionals of convolution type. Here, our functional $n \mapsto F(\varrho[n])$ is non-local and not of convolution type, and is much harder to analyze. We were not in particular able to prove the geodesic convexity, and therefore had to follow a different route. Note that it is mentioned in \cite{ambrosio}, page 290, that even in the simpler case of the Derrida-Lebowitz-Speer-Spohn equation  where the first variation of the functional is $-\Delta \sqrt{n}/\sqrt{n}$, it is not known if the functional $F$ has the geodesic convexity property.

We turn now to the semi-discretized version of \fref{qdd} introduced in \cite{QDD-SIAM}, which can be seen as the minimizing movement scheme of the theory of gradient flows. The discrete version will be the starting point of our analysis.

\paragraph{The semi-discretized equation.}
For $n_0$ given, the system reads
\begin{numcases}{} 
\frac{n_{k+1}-n_k}{\Delta t}+ \nabla \big(n_{k} \nabla (A_{k+1}-V_{k+1}) \big)=0\label{semiqdd}\\[3mm]
-\Delta V_{k+1}=n_{k+1}\label{poisson2} \\[3mm]
n_{k+1}=\sum_{p \in \Nm} e^{-\lambda_p[A_{k+1}]} |\phi_p[A_{k+1}]|^2\label{den2} 
\end{numcases}
where $(\lambda_p[A], \phi_p[A])_{p \in \Nm}$ are the spectral elements of the Hamiltonian $H[A]$ with the same domain as in \fref{domainH}. Solutions $A_{k}$ to \fref{semiqdd} are sought in $\Hp$, and those of  $\fref{poisson2}$ in $H^1_0(\Omega)$. Before stating an existence theorem for \fref{semiqdd}-\fref{poisson2}-\fref{den2}, we introduce the following functionals:
$$
\calF[n]=- \int_\Omega n (A[n]+1) dx+\frac{1}{2} \int_{\Omega} \left|\nabla V [n] \right|^2 dx, 
$$
which is formally equivalent to $F(\varrho[n])$, and 

\be \label{defSig}
\Sigma[n]= -\int_\Omega \big(n (A[n]-A[n_\infty]) +n-n_\infty \big)dx+\frac{1}{2} \int_{\Omega} \left|\nabla (V [n]-V[n_\infty]) \right|^2 dx,
\ee
which is essentially the relative entropy between $\varrho[n]$ and $\varrho_\infty$ (above $n_\infty=n[\varrho_\infty]$). Above, the equilibrium $\varrho_\infty$ is the solution to the minimization problem of Theorem \ref{SP} with constraint $\Tr(\varrho_\infty)= \|n_0\|_{L^1}$.

According to \cite{QDD-SIAM}, Theorem 3.1, the following result holds.
\begin{theorem} \label{thsemi} Let $n_0 \in C^0(\overline{\Omega})$ such that $n_0>0$ and $V_0 \in L^\infty(\Omega)$. Then, the system  \fref{semiqdd}-\fref{poisson2}-\fref{den2} admits a unique solution such that, for all $k \in \Nm$, $A_k \in \Hp$, $V_k \in H^1_0(\Omega)$ and $n_k \in C^0(\overline{\Omega})$ with $n_k>0$. We have moreover the following relations, for all $k \in \Nm$:
\begin{align}
\label{L1}&\int_{\Omega} n_k dx=\int_{\Omega} n_0 dx\\
\label{freeE} &\calF[n_k] + \Delta t\sum_{j=0}^{k-1} \int_\Omega n_{j} | \nabla (A_{j+1}-V_{j+1})|^2 dx \leq \calF[n_0]\\
\label{relat} & \Sigma[n_k] \leq \Sigma [n_0].
\end{align}
\end{theorem}

We present in the next section our main result, obtained in part by passing to the limit in \fref{semiqdd}-\fref{poisson2}-\fref{den2}.
\section{Main result} \label{main}


We define first the weak solutions to \fref{qdd} for an initial condition $n_0 \in L^2(\Omega)$: for $T>0$ arbitrary, we say that $(n,A,V)$ if a weak solution if $n \in L^2(0,T,L^2(\Omega))$, $A \in L^2(0,T,\Hp)$, $V \in L^2(0,T,H^1(\Omega))$, and if for any $\varphi \in C^1([0,T], \Hp)$ with $\varphi=0$ for $t \geq T$, we have
\be \label{weak1}
\int_0^T \big(n, \partial_t \varphi \big)dt + \big(n_0,\varphi(0)\big)+\int_0^T \big (n \nabla (A-V), \nabla \varphi \big) dt=0.
\ee
We introduce as well the relative entropy between two density operators $\varrho$ and $\sigma$:
$$
S(\varrho,\sigma)= \Tr \big( \varrho (\log \varrho - \log \sigma)\big) \in [0,\infty].
$$
Some properties of $S$ can be found e.g. in \cite{wehrl}. Our main result is the following.

\begin{theorem} \label{mainth} Let $n_0 \in \Hp$. Then, there exists $\delta>0$ such that the condition
\be
\label{closent}
\Sigma[n_0]=S(\varrho_0,\varrho_\infty)+\frac{1}{2} \|\nabla (V_0-V_\infty) \|^2_{L^2} \leq \delta
\ee
implies that the system \fref{qdd} admits a weak solution $(n,A,V)$, where $n \in L^\infty(0,T,\Hp)$, $\partial_t n \in L^2(0,T,\Hmp)$, $A \in L^2(0,T,\Hp)$, and $V \in L^\infty(0,T,H^1_0(\Omega))$. The associated quantum statistical equilibrium $\varrho:=\exp(-(H+A))$ satisfies $\varrho \in L^\infty(0,T,\calE_+)$ and $H_0 \varrho H_0 \in L^2(0,T,\calJ_1)$. The free energy satisfies moreover the relation, $t$ a.e., 
\be
\label{derivfree}
\frac{d}{dt} \calF[n(t)]=-\int_\Omega n(t) |\nabla (A(t)-V(t))|^2 dx.
\ee
Finally, the solutions converge exponentially fast to the equilibrium: there exists $\mu>0$ such that 
\be \label{expconv}
\calF[n(t)]-\calF[n_\infty] \leq \left(\calF[n(0)]-\calF[n_\infty]\right) e^{- \mu t}.
\ee
\end{theorem}

\bigskip 

Some comments are in order. First of all, this is only an existence result. For problems of the form \fref{GF}, the uniqueness is often a consequence of the geodesic convexity of the functional $F$, which is unknown at this point and explains in part the lack of a uniqueness result. Second, the condition \fref{closent} expresses that the initial state has to be sufficiently close to the equilibrium. It is a crucial point for the derivation of the bound from below for the density. The proof of the latter exploits the Sobolev embedding $H^1(\Omega) \subset L^\infty(\Omega)$, which is only valid in a one-dimensional setting. In higher dimensions, the condition \fref{closent} alone without the use of the embedding does not seem to be sufficient, and we are therefore limited to the 1D case since the bound from below is a key ingredient. Finally, the inequality \fref{expconv} implies the exponential convergence of $\varrho$ to $\varrho_\infty$ in $\calJ_2$. We have indeed, since $\varrho_\infty$ is a minimizer of the free energy under the global constraint,
$$
0 \leq F(\varrho(t))-F(\varrho_\infty) = \calF[n(t)]-\calF[n_\infty],
$$
and we will see further in Lemma \ref{FE}, in conjunction with the Klein inequality of Lemma \ref{pinsker}, that
\be \label{belF}
C \| \varrho(t)-\varrho_\infty\|^2_{\calJ_2} \leq S(\varrho(t),\varrho_\infty) \leq F(\varrho(t))-F(\varrho_\infty). 
\ee

The exponential convergence is obtained by deriving a non-commutative logarithmic Sobolev inequality in the spirit of \cite{carbone}. We will show that
\be \label{lsob}
F(\varrho)-F(\varrho_\infty) \leq C \| \sqrt{n} \nabla (A-V)\|^2_{L^2(\Omega)},
\ee
that can be recast in a more standard form as follows. For the sake of simplicity of the exposition, suppose that $\|n_0\|_{L^1}=1$, and therefore $\Tr(\varrho)=1$, and suppose as well that electrostatic effects can be neglected so the Poisson potential $V$ is zero. This implies in particular that $A_\infty$ is a constant. Introducing the operator $\calL=- [H,[H,\cdot]]$, a simple informal calculation based on the cyclicity of the trace and on the commutation between $H+A$ and $\varrho$ shows that (see \cite{QDD-JCP} for more details),
\bee
\| \sqrt{n} \nabla A \|^2_{L^2(\Omega)}&=&\| \sqrt{n} \nabla (A-A_\infty)\|^2_{L^2(\Omega)}\\
&=&-\Tr \big( (A-A_\infty) \calL \varrho \big)=\Tr \big( (\log \varrho - \log \varrho_\infty) \calL \varrho \big).
\eee
Together with \fref{belF} and \fref{lsob}, this leads to
$$
C S(\varrho ,\varrho_\infty) \leq \Tr \big( (\log \varrho - \log \varrho_\infty) \calL \varrho \big).
$$
When the latter holds for any density operator $\varrho$, the above inequality is referred to as a modified Log-Sobolev inequality of constant $C$ (for the operator $\calL$), see \cite{carbone}. Here, the inequality clearly does not hold for all $\varrho$, since any operator of the form $f(H)$ cancels the right-hand side (as $\calL f(H)=0$), leading to $\varrho_\infty=f(H)$ which is absurd when $f(x) \neq e^{-x}$. Note that the operator $\calL$ naturally arises in the derivation of QDD from the quantum Liouville equation, since the weak form of \fref{qdd} can be expressed formally as (when $V=0$),
$$
\Tr\big( \big[\partial_t \varrho - \calL \varrho \big] \varphi \big)=0, \qquad \forall \varphi.
$$
\medskip

The proof of theorem is decomposed into several steps. In section \ref{prelim}, we state various lemmas important for the proof. In section \ref{unif}, we derive a uniform bound from below for the density $n_k$ solution to the semi-discretized QDD. This leads to uniform bounds for $n_k$, $A_k$ and $V_k$, which allow us, using classical compactness arguments, to pass to the limit in \fref{semiqdd}  and to recover \fref{qdd}. This is done in section \ref{lim1}. Obtaining the closure relation \fref{clos} is the more difficult and interesting part. This is done by in section \ref{lim2} by deriving some stability estimates for local minimizers of the form of Theorem \ref{local}, and by using the representation formula \fref{repform}. Finally, the exponential convergence is addressed in section \ref{expoconv}, and is a consequence of  the inequality \fref{lsob} and the dissipation of the free energy \fref{derivfree} proved in section \ref{proofderiv}.

\section{Proof of the theorem} \label{proofth}

We start with a series of technical lemmas that will be used throughout the proof.
\subsection{Preliminary technical results} \label{prelim}

The first lemma below is crucial and provides us with a lower bound for the relative entropy. It is taken from \cite{LS}, Theorem 3. 
\begin{lemma} \label{pinsker} (Klein inequality). For all $\varrho_1$ and $\varrho_2$ in $\calE_+$, we have
$$
C \Tr \big((1+|\log \varrho_2|)(\varrho_1-\varrho_2)^2 \big)\leq S(\varrho_1,\varrho_2), 
$$
where $C$ is independent of $\varrho_1$ and $\varrho_2$
\end{lemma}

The next lemma, proved in \cite{MP-JSP}, shows that the entropy is well defined for density operators in $\calE_+$.

\begin{lemma} \label{below} There exists $C>0$, such that, for all $\varrho \in \calE_+$,
\begin{align} \label{ent1}
&-\big(\Tr \sqrt{H_0} \varrho \sqrt{H_0} \big)^{1/2} \leq C \Tr \big(\beta(\varrho)\big)\\ \label{ent2}
& \Tr \big(|\beta(\varrho)|\big) \leq C\|\varrho\|_\calE.
\end{align}
\end{lemma}

We will need as well the following Lieb-Thirring type inequalities (the elementary proof can be find in \cite{MP-JSP}, Lemma 5.3):
\begin{lemma} \label{lieb} Suppose $\varrho$ is self-adjoint and belongs to $\calE$. Then, the following estimates hold:
\begin{align}
\label{ninfty}\|n[\varrho]\|_{C^0(\overline{\Omega})} \leq C \| \varrho\|^{1/4}_{\calJ_2}  \|\varrho\|_{\calE}^{3/4}\\
\label{gradnl2}\|\nabla n[\varrho]\|_{L^2} \leq C \| \varrho\|^{1/4}_{\calJ_1}  \|\varrho\|_{\calE}^{3/4}.
\end{align}
\end{lemma}

The lemma below  provides us with additional regularity on the local minimizer knowing the potential in the Hamiltonian is in $L^2(\Omega)$. This will be an important point in the identification of the limit version of \fref{den2}. The proof is given in the Appendix. 
\begin{lemma} \label{regmin} (Regularity of the minimizer). Let  $V \in L^2(\Omega)$  and define $\varrho=\exp(-(H_0+V))$. Then $\varrho \in \calE_+$ and $H_0 \varrho H_0 \in \calJ_1$, with the estimate
\be \label{estHrH}
\Tr\big( H_0 \varrho H_0 \big) \leq
C+ C\big(1+\|V\|^2_{L^2}\big) \|\varrho\|_\calE.
\ee
\end{lemma}

The next two lemmas are classical results about eigenvalues and eigenvectors of density operators. The proofs can be found for instance in \cite{MP-JSP}, Lemmas A.1 and A.2.

\begin{lemma} \label{lieb2} Let $\varrho \in \calE_+$ and denote by $(\rho_p)_{p \in \Nm}$ the eigenvalues of $\varrho$ (nonincreasing and counted with multiplicity), associated to the orthonormal family of eigenfunctions $(\phi_p)_{p\in \Nm}$. Denote by $(\lambda_p[\calH])_{p \in \Nm}$ the eigenvalues of some Hamiltonian $\calH$ with compact resolvant. Then we have
$$
\Tr \big(\sqrt{\calH} \varrho \sqrt{\calH}\big)=\sum_{p \in \Nm} \rho_p \,\big(  \sqrt{\calH} \phi_p, \sqrt{\calH}\phi_p\big) \geq \sum_{p \in \Nm} \rho_p \,\lambda_p[\calH].
$$
\end{lemma}

\begin{lemma} \label{lipspec} Let $\varrho$ and $\varrho_N$ be two nonnegative trace class operators such that $\varrho_N$ converges strongly to $\varrho$ in $\calL(L^2(\Omega))$, and  denote by $(\rho_p)_{p \in \Nm}$ and $(\rho_p^N)_{p \in \Nm}$ the eigenvalues of $\varrho$ and $\varrho_N$. Then, there exist a sequence of orthonormal eigenbasis $(\phi_p^N)_{p \in \Nm}$ of $\varrho_N$, and an orthonormal eigenbasis $(\phi_p)_{p \in \Nm}$ of $\varrho$, such that, 
$$
\forall p\in \Nm,\qquad \lim_{N\to \infty} \rho_p^N =\rho_p, \qquad  \lim_{N\to \infty}\|\phi_p^N-\phi_p\|_{L^2(\Omega)}=0.
$$
\end{lemma}


The last lemma allows us to identify the free energies $F(\varrho)$ and $\calF[n]$ when $\varrho$ is the minimizer of $F$ and when $n=n[\varrho]$. We obtain as well some bounds on the relative entropy between two minimizers in terms of the difference of their associated potentials and densities.

\begin{lemma} \label{lemlip} Let $A \in L^2(\Omega)$, and define $\varrho=\exp(-(H+A))$ with the notation $n\equiv n[\varrho]$. Let morever $W_1$ and $W_2$ be in $L^2(\Omega)$, with $\varrho_i=\exp(-(H_0+W_i))$, $n_i\equiv n[\varrho_i]$, $i=1,2$. Denote by $(n_\infty, A_\infty, V_\infty)$ the solution to the global minimization problem of Theorem \ref{SP} with $N=\Tr(\varrho)$. For $F(\varrho)$ the free energy of $\varrho$ defined in \fref{deffreeE}, $V\equiv V[n]$ the Poisson potential, $\Sigma[n]$ defined in \fref{defSig} and $S(\varrho_i,\varrho_j)$ the relative entropy between $\varrho_i$ and $\varrho_j$, we have the relations:
\begin{align}
&F(\varrho)=-(A+1,n)+\frac{1}{2}\|\nabla V\|^2_{L^2}  \label{freeA}\\
&\Sigma[n]= S(\varrho,\varrho_\infty)+\frac{1}{2}\|\nabla (V-V_\infty)\|^2_{L^2}  \label{rela}\\
&S(\varrho_1,\varrho_2)+S(\varrho_2,\varrho_1)\leq (W_2-W_1,n_1-n_2). \label{rel}
\end{align}
\end{lemma}
\begin{proof} The proof essentially consists in justifying the straightforward formal calculations. According to Lemma \ref{regmin}, we have $\varrho \in \calE_+$ and  $H_0 \varrho H_0 \in \calJ_1$ since $V_0+A \in L^2(\Omega) $, and therefore
$$
\Tr \big(\sqrt{H_0} \varrho \sqrt{H_0}\big)=\Tr \big(\varrho H_0\big).
$$
In the same way, the above regularity ensures that
$$
\Tr \big( \varrho  \log \varrho\big)=-\Tr \big(\varrho (H_0+V_0+A)\big),
$$
which leads to \fref{freeA} following the definition of $F(\varrho)$. Regarding \fref{rel}, we have again from Lemma \ref{regmin} that $\varrho_i \in \calE_+$ and $H_0 \varrho_i H_0 \in \calJ_1$. This implies in particular from \fref{ninfty} that $n[\varrho_i] \in L^\infty(\Omega)$, and from \fref{ent2} that $\Tr \big(\varrho_i \log \varrho_i \big)$ is finite. Similarly, since $\log \varrho_j=-H_0-W_j$, $\Tr \big(\varrho_i \log \varrho_j \big)$ is finite. As a consequence
\bee
S(\varrho_i, \varrho_j)&=& \Tr \big( \varrho_i (\log \varrho_i - \log \varrho_j)\big)\\
&=&- (n_i,W_i-W_j),
\eee
which leads to \fref{rel} summing $S(\varrho_1, \varrho_2)$ and $S(\varrho_2, \varrho_1)$. The relation \fref{rela} is obtained by setting $W_i=V_0+A$ and $W_j=V_0+A_\infty$ (which is in $L^2(\Omega)$ according to Theorem \ref{SP}) in the latter equation and by identifying with \fref{defSig}.
\end{proof}
\medskip

The first step of the proof is to obtain uniform estimates for the semi-discrete problem \fref{semiqdd}=\fref{poisson2}-\fref{den2}.

\subsection{Uniform estimates for the semi-discrete problem} \label{unif}

Let $n_0 \in \Hp \subset C^0(\overline{\Omega})$ be the initial density of Theorem \ref{mainth}. We will prove in Proposition \ref{propbelow} further that the condition \fref{closent} on the initial relative entropy $S(\varrho_0,\varrho_\infty)$ implies that $n_0(x)>0$ on $\overline{\Omega}$. As a consequence, we obtain from Theorem \ref{thsemi} a unique sequence of solutions $(n_k,A_k,V_k)_{k \in \Nm}$ to the semi-discretized problem. We define then $\varrho_k:=\exp(-(H+A_k))$, which belongs to $\calE_+$ according to Lemma \ref{regmin} since $A_k \in \Hp$. By construction, we have $n[\varrho_k]=n_k$, and according to Theorem \ref{local}, $\varrho_k$ is the unique minimizer of the free energy $F(\varrho)$ under the constraint $n[\varrho]=n_k$. We introduce  moreover the notation $F_k:=F(\varrho_k)$. Note that since $H_0 \varrho_k H_0 \in \calJ_1$ by Lemma \ref{regmin}, we have $F_k=\calF[n_k]$ according to \fref{freeA}, which will be used throughout the proof. 

We have then the following lemma:

\begin{lemma} \label{estimrho} The solution to the semi-discretized system \fref{semiqdd}-\fref{poisson2}-\fref{den2} satisfies, $\forall k \in \Nm$,
\begin{align}
& \label{rhoE}\|\varrho_k \|_{\calE} \leq C\\
&\label{nH1} \|n_k \|_{H^1} \leq C\\
&\label{Fk} |F_k| \leq C,
\end{align}
for a constant $C$ independent of $k$ and $\Delta t$.
\end{lemma}
\begin{proof} First, we have from the decay of the free energy stated in \fref{freeE} that $F_k \leq F_0$, $\forall k \in \Nm$. Then, estimate \fref{ent1} and the Young inequality yield from the definition of $F(\varrho_k)$,
$$
C\Tr \big (\sqrt{H_0} \varrho_k \sqrt{H_0} \big)+\Tr\big (V_0 \varrho_k \big) \leq F_k \leq F_0.
$$
Since $V_0 \in L^\infty(\Omega)$, the trace term involving $V_0$ can be bounded by 
$$\|V_0\|_{L^\infty} \|\varrho_k \|_{\calJ_1}=\|V_0\|_{L^\infty} \|n_k \|_{L^1}=\|V_0\|_{L^\infty} \|n_0 \|_{L^1}$$
 thanks to the conservation of the $L^1$ norm of $n_k$ given in \fref{L1}. This proves \fref{rhoE} together with $\Tr(\varrho_k)=\|n_0 \|_{L^1}$. Estimate \fref{nH1} is a consequence of \fref{rhoE}, \fref{ninfty}, and \fref{gradnl2}. The bound \fref{Fk} is direct since following \fref{ent1} and \fref{rhoE},
$$
-C \leq -C \left(\Tr \big (\sqrt{H_0} \varrho_k \sqrt{H_0} \big) \right)^{1/2} - \|V_0\|_{L^\infty} \|n_0 \|_{L^1} \leq F_k \leq F_0.
$$
This ends the proof.
\end{proof}

The next proposition is crucial, and provides us with a uniform bound from below for the local density $n_k$.

\begin{proposition} \label{propbelow} (Bound from below for the density) Let $n_0 \in \Hp$, nonnegative. Then, for the $\Sigma[n_0]$ defined in \fref{closent}, there exists $\delta>0$, and $\underline{n}>0$ independent of $k$ and $\Delta t$ such that the condition
$$
\Sigma[n_0] \leq \delta
$$
implies
$$
n_{k}(x) \geq \underline{n}, \qquad \forall k \in \Nm, \qquad \forall x \in \overline{\Omega}.
$$

\end{proposition}
\begin{proof} Suppose first that the lower bound is satisfied for $k=0$, so that the hypotheses of Theorem \ref{thsemi} on the initial condition $n_0$ hold. We treat the case $k=0$ at the end of the proof.  The key fact is then that the first term in the definition of $\Sigma[n_k]$ is the relative entropy between $\varrho_k$ and $\varrho_\infty$, that is, according to \fref{rela} and \fref{L1},
$$
\int_{\Omega} \left(n_k (A_k-A_\infty)+n_k-n_\infty \right) dx= S(\varrho_k,\varrho_\infty)=\Tr \big( \log (\varrho_k) (\varrho_k - \varrho_\infty) \big).
$$
The inequality of Lemma \ref{pinsker} then implies that, for all $k \in \Nm^*$,
$$
\|\varrho_k -\varrho_\infty\|_{\calJ_2}^2 \leq C S(\varrho_k,\varrho_\infty).
$$
Since $\Sigma[n_k] \leq \Sigma[n_0]$ for all $k$ according to \fref{relat}, we conclude from $\Sigma[n_0] \leq \delta$ that
$$
\|\varrho_k -\varrho_\infty\|_{\calJ_2}^2 \leq C \delta.
$$
Besides, \fref{ninfty} yields 
\be \label{Dn}
\| n_k-n_\infty\|_{C^0(\overline{\Omega})} \leq C \|\varrho_k -\varrho_\infty\|_{\calJ_2}^{1/4} \left(\|\varrho_k\|_{\calE}+\|\varrho_\infty\|_\calE \right)^{3/4},
\ee
which together with \fref{rhoE} and Theorem \ref{SP} for the fact that $\varrho_\infty \in \calE_+$, leads to 
$$
\| n_k-n_\infty\|_{C^0(\overline{\Omega})} \leq C \delta^{1/8}, \qquad \forall k \in \Nm^*.
$$
Since finally $n_\infty(x) \geq \underline{n}_\infty$ on $\overline{\Omega}$ according to Theorem \ref{SP}, this concludes the proof for $k \in \Nm^*$ by setting $\delta$ sufficiently small.

When $k=0$, a similar argument carries over: let $\varrho_0$ be the unique solution to the minimization problem of Theorem \ref{local} with constraint $n_0$. According to this latter theorem, $\varrho_0$ belongs to $\calE_+$. Then \fref{Dn} holds, and the lower bound is obtained as above. This ends the proof.
\end{proof}

We will need in addition the following bound on $A_k$.

\begin{corollary} \label{coro} We have the estimate 
$$
 \|A_k \|_{\Hmp} \leq C, \qquad \forall k \in \Nm,
$$
for a constant $C$ independent of $k$ and $\Delta t$.
\end{corollary}
\begin{proof} The estimate is a consequence of the representation formula \fref{repform}: we have first, thanks to Proposition \ref{propbelow}, Lemma \ref{below}  and \fref{rhoE},
$$
\left\|\frac{n[\nabla \varrho_k \nabla] -n[\varrho_k \log \varrho_k] }{n_k} \right\|_{L^1} \leq \frac{1}{\underline{n}} \left(\Tr \big( \sqrt{H_0} \varrho_k \sqrt{H_0}\big)+\Tr \big( |\beta(\varrho_k)|\big)+\Tr \big( \varrho_k\big) \right) \leq C.
$$
On the other hand,
$$
\left\|\frac{\Delta n_k}{n_k} \right\|_{\Hmp} \leq \frac{1}{\underline{n}} \|n_k\|_{H^1}+\frac{1}{\underline{n}^2} \|n_k\|^2_{H^1} \leq C,
$$
by Proposition \ref{propbelow} and \fref{rhoE}. This, together with $V_0 \in L^\infty(\Omega)$, concludes the proof.
\end{proof}
\medskip

The next step of the proof is to define approximations of $(n,A,V)$ from $(n_k,A_k,V_k)_{k \in \Nm}$ and to pass to the limit.

\subsection{Passing to the limit}  Let $T>0$, and for $N \geq 1$, set $\Delta t= T/N$ in \fref{semiqdd}. We define $\hat n_{N}(t,x)$, $\hat A_N(t,x)$ and $\hat V_{N}(t,x)$ as the following piecewise constant functions, for $(t,x) \in (0,T) \times \Omega$:
$$
\left.
\begin{array}{l}
\ds  \hat n_{N}(t,x)= n_{k+1}(x)\\
\ds  \hat A_{N}(t,x)= A_{k+1}(x)\\
\ds  \hat V_{N}(t,x)= V_{k+1}(x)
\end{array}
\right\}
 \qquad \textrm{when } t \in (k \Delta t, (k+1) \Delta t], \qquad k=0,\cdots,N-1.
$$
We define in the same way the operator $\hat \varrho_N(t)$. From the semi-discretized equation \fref{semiqdd}, we deduce that the functions $\hat n_N$, $\hat V_N$ and $\hat A_N$ satisfy, for all $\varphi \in C^1([0,T],\Hp)$, with $\varphi=0$ for $t \geq T$,
\begin{align*} 
&\frac{1}{\Delta t}\int_0^T \big(\hat n_{N}(t)-\hat n_{N}(t-\Delta t),\varphi(t) \big) dt \\ \nonumber
&\hspace{2cm}=\int_0^T \big (\hat n_{N}(t-\Delta t) \nabla (\hat A_{N}(t)-\hat V_{N}(t)), \nabla \varphi(t) \big) dt.
\end{align*}
We recast the left-hand side as 
\begin{align} \label{disPDE}
&\frac{1}{\Delta t}\int_0^T \big(\hat n_{N}(t)-\hat n_{N}(t-\Delta t),\varphi(t) \big) dt \\ \nonumber
&\hspace{2cm}=\frac{1}{\Delta t}\int_0^{T-\Delta t } \big(\hat n_{N}(t),\varphi(t)- \varphi(t+\Delta t) \big) dt\\ \nonumber
&\hspace{2cm}\qquad +\frac{1}{\Delta t}\int_{T-\Delta t}^{T} \big(\hat n_{N}(t),\varphi(t))dt-\frac{1}{\Delta t}\int_{-\Delta t}^{0} \big(\hat n_{N}(t),\varphi(t+\Delta t))dt\\[3mm]
&\hspace{2cm}:=T^N_1+T^N_2+T^N_3. \nonumber
\end{align}

In the last term above, we include the initial condition by replacing $\hat n_{N}(t)$ by $n_0$ for $t \in (-\Delta t,0)$. The first step is to pass to the limit in \fref{disPDE} in order to recover the weak formulation \fref{weak1}. We will need for this the estimates given the next proposition.
\begin{proposition} The following uniform bounds hold:
\begin{align}
\label{bound1}&\| \hat n_N \|_{L^\infty(0,T,H^1)} \leq C\\
\label{bound3}&\underline{n} \leq \hat n_N(t,x), \qquad \forall (t,x)\in (0,T)\times \overline{\Omega}\\
\label{bound5}& \| \hat V_{N} \|_{L^\infty(0,T,H^2)} \leq C\\
\label{bound4b}& \| (\hat A_{N})_\Omega \|_{L^2(0,T)} \leq C\\
\label{bound4}& \| \hat A_{N} \|_{L^2(0,T,H^1)} \leq C\\
\label{bound7}& \| \hat \varrho_N \|_{L^\infty(0,T,\calE)} \leq C\\
\label{bound6}& \| H_0\, \hat \varrho_{N} \, H_0 \|_{L^2(0,T,\calJ_1)} \leq C.
\end{align}
Above, $(\hat A_{N})_\Omega$ denotes the average of $\hat A_N$ over $\Omega$, i.e. $(\hat A_{N})_\Omega=\int_0^1 \hat A_{N}(x)dx$. Moreover, for any $h \in (-1,1)$, and for $\tau_h \hat n_N(t,x):=\hat n_N(t+h,x)$, with extension $\hat n_N(t,x)=0$ if $t \notin (0,T]$, we have
\be
\label{bound2}\| \tau_h \hat n_{N}-\hat n_N \|_{L^1(0,T,\Hmp)} \leq C \sqrt{h}.
\ee
\end{proposition}
\begin{proof} Estimate \fref{bound1} follows from \fref{nH1} and
$$
\| \hat n_{N} \|_{L^\infty(0,T,H^1)} \leq \sup_{k \in \Nm} \| n_k \|_{H^1}.
$$
The bound from below \fref{bound3} is a consequence of Proposition \ref{propbelow}.
Estimate \fref{bound5} follows from \fref{bound1}, the Poisson equation \fref{poisson2}, and elliptic regularity. Estimate \fref{bound4b} is obtained by noticing that
$$
F_k=-(A_k)_\Omega \int_\Omega n_k dx -\int_\Omega n_k \big(A_k-(A_k)_\Omega+1\big) dx+\frac{1}{2} \int_{\Omega} \left|\nabla V_k \right|^2 dx.
$$
Dividing by $\|n_k\|_{L^1}$, the latter equation, together with \fref{L1} and \fref{Fk}, shows that
\bee
|(A_k)_\Omega | &\leq & C \| \nabla V_k \|_{L^2}^2+C \|A_k-(A_k)_\Omega\|_{L^\infty}+C\\
&\leq & C+C\|\nabla A_k \|_{L^2}\\
&\leq & C+C\| \sqrt{n_{k-1}}\nabla A_k \|_{L^2}.
\eee
In the second line above, we used the fact that $V_k$ is uniformly bounded in $H^1(\Omega)$ thanks to the Poisson equation and \fref{rhoE}, together with the fact that $H^1(\Omega) \subset L^\infty(\Omega)$ and the Poincar\'e-Wirtinger inequality in order to control $A_k-(A_k)_\Omega$ in $L^2(\Omega)$ by its gradient in $L^2(\Omega)$. In the last line, we used the lower bound of Proposition \ref{propbelow}. Then, 
\bee
\|(\hat A_N)_\Omega\|^2_{L^2(0,T)} &\leq& C+C \Delta t \sum_{j=0}^{N-1} \int_\Omega n_{j} |\nabla (A_{j+1}-V_{j+1})|^2 dx\\
&&\qquad +C\|\hat n_N\|_{L^1(0,T,L^\infty)} \|\nabla \hat V_N\|_{L^\infty(0,T,L^2)} \\
& \leq& C,
\eee
thanks to \fref{freeE}, \fref{bound1}, and \fref{bound5} for controlling the Poisson potential. Estimate \fref{bound4} follows then from the bound \fref{freeE} on the free energy, the lower bound \fref{bound3}, the bound \fref{bound5} for the Poisson potential, and a combination of \fref{bound4b} and the Poincar\'e-Wirtinger inequality. Estimate \fref{bound7} is a consequence of \fref{rhoE}. Regarding \fref{bound6}, we remark first that $\| \hat A_N\|_{L^\infty(0,T,\Hmp)} \leq C$ according to Corollary \ref{coro}. Together with \fref{bound4} and standard interpolation, we can conclude that $\| \hat A_N\|_{L^4(0,T,L^2)} \leq C$. The result then follows from \fref{estHrH} and \fref{rhoE}.

We turn now to estimate \fref{bound2}. Let $t \in (0,T)$ with $t \notin U:=\{ k T /N, \; N \in \Nm^*, k=0, \cdots, N\}$, and let first $h \in [0,1)$. Write then $t= k_1 \Delta t +r_1$ and $h=k_2 \Delta t +r_2$, where $k_1$ and $k_2$ are integers, and where $r_1 \in (0,\Delta t)$, $r_2 \in [0, \Delta t)$. When $r_1+r_2 \leq \Delta t$, we have, for any $\varphi \in C^1([0,T],\Hp)$, 
$$
I_N(t,h)(\varphi):=(\hat n_{N}(t+h)-\hat n_N(t),\varphi)=(n_{k_1+1+k_2}-n_{k_1+1},\varphi),
$$
with $n_{k}=0$ for $k>N$, while when $r_1+r_2 > \Delta t$, we have
$$
 I_N(t,h)(\varphi)=(n_{k_1+2+k_2}-n_{k_1+1},\varphi).
$$
Let us start with the case $r_1+r_2 \leq \Delta t$. Using \fref{semiqdd}, we can recast $I_N$ as the telescopic sum (below, $a \wedge b=\min(a,b)$),
\bee
I_N(t,h)(\varphi)&=&\sum_{p=k_1+1}^{(k_1+k_2) \wedge N}(n_{p+1}-n_{p},\varphi)=\Delta t \sum_{p=k_1+1}^{(k_1+k_2)\wedge N}(n_{p} \nabla (A_{p+1}-V_{p+1}),\nabla \varphi).
\eee
Hence, the Cauchy-Schwarz inequality leads to
\bee
|I_N(t,h)(\varphi)| &\leq&   \sqrt{ k_2 \Delta t} \, \|\nabla \varphi \|_{L^2} \sup_{p=1,\cdots,N}\|\sqrt{n_p}\|_{L^\infty} \\
&&\times \left(\Delta t\sum_{p=k_1+1}^{(k_1+k_2)\wedge N}\| \sqrt{n_p} \nabla (A_{p+1}-V_{p+1})\|^2_{L^2}\right)^{1/2}.
\eee
Together with \fref{freeE} and \fref{nH1}, this yields by duality, since $k_2 \Delta t \leq h$,
$$
\| \hat n_{N}(t+h)-\hat n_N(t) \|_{\Hmp} \leq C \sqrt{ h}, \qquad \forall N \in \Nm^*,
\qquad \forall h \in [0,1),$$
which holds for all $t \notin U$. Note that the latter inequality cannot hold for $t \in U$ since when $t=\frac{k T}{N}$ for instance, then $\hat n_N(t)=n_{k}$ while $\hat n_N(t+h)=n_{k+1}$ for $h \in (0,\Delta t)$. We find the same estimate when $r_1+r_2 > \Delta t$. Since $U$ is countable and therefore of Lebesgue measure zero, we obtain \fref{bound2} when $h \geq 0$. The case $h <0$ follows similarly. This ends the proof of the proposition.
\end{proof}

\subsubsection{Passing to the limit in the weak formulation \fref{disPDE}} \label{lim1} We have now all the required estimates to obtain \fref{weak1}.

\paragraph{Compactness.} For a finite constant $C>0$ and $|h|<1$, $h\neq 0$, let $S$ be the set defined by 
$$
 S=\big\{ u \in L^\infty(0,T,L^2(\Omega)) , \quad \| u \|_{L^\infty(0,T,H^1)} +h^{-1/2} \| \tau_h u \|_{L^1(0,T,\Hmp)} \leq C\big\},
$$
where $\tau_h u(t):=u(t+h)$.
Then $S$ is relatively compact in $L^2(0,T,L^2(\Omega))$ as an application of the Riesz-Fr\'echet-Kolmogorov criterion: Indeed, for $(h_t,h_x) \in (-1,1) \times (-1,1)$, let $\calT_{h_x,h_t}u(t,x):=u(t+h_t,x+h_x)$, where $u$ is extended by zero when $t+h_t \notin (0,T)$. Then,
\be \label{Th}
\| \calT_{h_t,h_x}u-u\|_{L^2(\Omega)} \leq \| \calT_{h_t,h_x}u-\calT_{h_t,0}u\|_{L^2(\Omega)}+\| \calT_{h_t,0}u-u\|_{L^2(\Omega)}.
\ee
For the last term, we write
\bee
\| \calT_{h_t,0}u-u\|^2_{L^2(\Omega)}&=& \left( \calT_{h_t,0}u-u, \calT_{h_t,0}u -u \right) \leq \|\calT_{h_t,0}u-u\|_{\Hmp} \|\calT_{h_t,0}u-u\|_{H^1}\\
&\leq& 2 \|u\|_{L^\infty(0,T,H^1)} \|\calT_{h_t,0}u-u\|_{\Hmp}.
\eee
Integrating in time and using the bounds given in the definition of $S$, this yields
\be \label{kol}
 \| \calT_{h_t,0}u-u\|_{L^2(0,T,L^2)} \leq C |h_t|^{1/4}, \qquad \forall u \in S.
\ee
The remaining term in \fref{Th} is standard owing to the $H^1$ regularity in the spatial variable and we find
$$
\| \calT_{h_t,h_x}u- \calT_{h_t,0}u\|_{L^2(0,T,L^2)} \leq C |h_x|^{1/2}, \qquad \forall u \in S.
$$
This shows the relative compactness of $S$ in $L^2(0,T,L^2(\Omega))$. 

Now, according to \fref{bound1} and \fref{bound2}, the sequence $(\hat n_{N})_{N \in \Nm^*}$ belongs to $S$ for an appropriate $C$. There exists therefore $n \in L^2(0,T,L^2(\Omega))$, and a subsequence (still denoted by $(\hat n_{N})_{N \in \Nm^*}$; this abuse of notation will consistently be done with any subsequences), such that $\hat n_{N} \to n$ strongly in $n \in L^2(0,T,L^2(\Omega))$. The bound \fref{bound1} implies moreover that $n \in L^\infty(0,T,\Hp)$, and \fref{kol} with $u \equiv \hat n_N$ shows that $\hat n_N$ and $\tau_{\Delta t} \hat n_N$ have the same strong limit. Furthermore, we conclude from \fref{bound4} that, along subsequences, $\hat A_{N} \to A$ weakly in $L^2(0,T,\Hp)$ for some $A \in L^2(0,T,\Hp)$, and from \fref{bound5} that $\hat V_{N} \to V$ weakly-$*$ in $L^\infty(0,T,H^{1}_0(\Omega))$ for some $V \in L^2(0,T,H^{1}_0(\Omega))$. We can now pass to the limit in \fref{semiqdd}

\paragraph{The limit.} According to what we have found above, we have, for all $\varphi \in C^2([0,T],\Hp)$, with $\varphi=0$ for $t \geq T$,
$$
\lim_{N \to \infty} \int_0^T \big (\hat n_{N}(t-\Delta t) \nabla (\hat A_{N}(t)-\hat V_{N}(t)), \nabla \varphi(t) \big) = \int_0^T \big ( n (t) \nabla ( A(t)- V(t)), \nabla \varphi(t) \big).
$$
It remains to treat the terms $T_1^N$, $T_2^N$ and $T_3^N$ in \fref{disPDE}. We have for $T_1^N$, and some $t_0(t) \in (t,t+\Delta t)$:
\bee
T_1^N &=& - \int_0^{T-\Delta t} (\hat n_N(t), \partial_t \varphi(t))dt- \frac{\Delta t}{2}\int_0^{T-\Delta t} (\hat n_N(t), \partial^2_{tt} \varphi(t_0(t)))dt.
\eee
The first term converges to
$$
- \int_0^{T} ( n(t), \partial_t \varphi(t))dt
$$
since $\hat n_N$ converges to $n$ strongly in  $L^2(0,T,L^2(\Omega))$ as mentioned before, while the second one converges to zero thanks to \fref{bound1}. For $T_2^N$, we use the fact that $\varphi(t)=\partial_t \varphi(t)=0$ for $t \geq T$, which leads to, for some $t_1(t) \in (t,T)$,
\bee
T_2^N &=& \frac{1}{\Delta t}\int_{T-\Delta t}^T (\hat n_N(t), \varphi(t)-\varphi(T))dt\\
&=&\frac{1}{2 \Delta t} \int_{T-\Delta t}^T (t-T)^2(\hat n_N(t), \partial_{tt} \varphi(t_1(t))dt.
\eee
The last term can be controlled by
$$
C \Delta t \|\hat n_N \|_{L^\infty(0,T,L^2)}, 
$$
and therefore goes to zero as $N \to \infty$. The term $T_3^N$ is straightforward and yields 
$$
\lim_{N \to \infty} T_3^N= -(n_0,\varphi(0)).
$$
We therefore recover the weak formulation \fref{weak1}. The lower bound on the density is obtained as follows: from the strong convergence of $\hat n_N$ in $L^2(0,T,L^2(\Omega))$, we deduce that there exists a subsequence such that $\hat n_N \to n$ almost everywhere in $(0,T) \times \Omega$. Passing to the limit in \fref{bound3} leads to  
$$
\underline{n} \leq n, \qquad a.e. \quad
 (0,T) \times \Omega.
$$
Finally, the fact that $\partial_t n \in L^2(0,T,\Hmp)$ follows directly by duality since, for any $\varphi$ smooth supported in $(0,T)$,
\bee
\left| \int_0^T (n,\partial_t \varphi) dt\right| &\leq& \left| \int_0^T (n \nabla (A-V),\nabla \varphi) dt\right| \\
&\leq& \|n\|_{L^\infty(0,T,L^\infty)} \|\nabla (A-V)\|_{L^2(0,T,L^2)} \|\nabla \varphi\|_{L^2(0,T,L^2)}\\
&\leq& C \|\nabla \varphi\|_{L^2(0,T,L^2)}.
\eee
\medskip

We turn now to the derivation of the relation between the limiting $n$ and $A$, which is the most delicate part of the proof.
\subsubsection{Passing to the limit in the closure relation \fref{den2}} \label{lim2}

The proof requires more work than the previous direct limit. The starting point is to consider  $\hat \varrho_{N}=\exp(-(H+\hat A_{N}))$, which is the solution to the local minimization problem with constraint $\hat n_N$. For $A$ and $n$ the limits of $\hat A_N$ and $\hat n_N$ obtained before, the goal is to show that $\hat \varrho_N$ converges to $\exp(-(H+A))$, where $n[\exp(-(H+A))]=n$. The main difficulty is that we only have weak convergence of $\hat A_{N}$ with respect to the time variable. With at least almost sure convergence in time and strong convergence in space, it would be direct to conclude from classical perturbation theory that the expected limit holds. Here, we need to proceed differently, and the key ingredients are the representation formula \fref{repform} and the stability estimate \fref{rel}. The latter allows us (i) to show that the limit of $\hat \varrho_N$ is of the form $\exp(-(H+A_{eq}))$ and (ii) to transfer the strong convergence in time of $\hat n_N$ to $\hat \varrho_{N}$, while the former allows to conclude that $A_{eq}=A$.

The first step is to obtain more compactness results.

\paragraph{Step 1: more compactness.} We deduce first from \fref{L1} that 
$$
\|\hat \varrho_{N}\|_{L^\infty(0,T,\calJ_1)}=\|\Tr \big( \hat \varrho_{N} \big)\|_{L^\infty(0,T)}=\|\hat n_{N}\|_{L^\infty(0,T,L^1)}=\|n_{0}\|_{L^1}.
$$
We can therefore extract a subsequence such that $\hat \varrho_{N} \to \varrho $ weakly-$*$ in $L^\infty(0,T,\calJ_1)$. In the same way, we conclude from \fref{bound7} that $\sqrt{H_0} \hat \varrho_{N} \sqrt{H_0} \to \sqrt{H_0} \varrho \sqrt{H_0} $ weakly-$*$ in $L^\infty(0,T,\calJ_1)$, and  from \fref{bound6} that $H_0 \hat \varrho_{N} H_0 \to H_0 \varrho H_0 $ weakly-$*$ in $L^2(0,T,\calJ_1)$. As claimed in Theorem \ref{mainth}, if we assume temporarily that $\varrho=\exp(-(H+A))$, this shows in particular that $\varrho \in L^\infty(0,T,\calE_+)$ and that $H_0 \varrho H_0 \in L^2(0,T,\calJ_1)$.

Consider then the limit $n$ of $\hat n_N$, which  belongs to $L^\infty(0,T,\Hp)$, and verifies $\underline{n} \leq n$, a.e. in  $(0,T) \times \Omega$. According to Theorem \ref{local}, the free energy $F(u)$ admits a unique minimizer under the local density constraint $n[u]=n(t)$,  $t$ almost everywhere. We denote by $\varrho_{eq}[n(t)] \in \calE_+$ the solution, and by $A_{eq}(t) \in \Hmp$ the corresponding Lagrange multiplier. We want to show that $A_{eq}=A$. A step towards this is to show that $\varrho_{eq}[n]=\varrho$, which is a consequence of the following lemma, proved at the end of the section.

\begin{lemma} \label{lemL2} The chemical potential $A_{eq}$ belongs to $L^2(0,T,L^2(\Omega))$.
\end{lemma}

The latter lemma allows us to use the stability estimate \fref{rel} together with Lemma \ref{pinsker} to conclude that, $t$ a.e.,
$$
\Tr \big( \varrho_{eq}[n(t)]-\hat \varrho_N(t)  \big)^2 \leq C (A_{eq}(t)- \hat A_N(t),\hat n_N(t)-n(t)).
$$
Integrating in time, we find
\be \label{strongL2}
\| \varrho_{eq}[n]-\hat \varrho_N \|^2_{L^2(0,T,\calJ_2)} \leq \big( \|A_{eq}\|_{L^2(0,T,L^2)}+ \|\hat A_N\|_{L^2(0,T,L^2)} \big) \|\hat n_N-n\|_{L^2(0,T,L^2)}.
\ee
The strong convergence of $\hat n_N$ to $n$ in $L^2(0,T,L^2(\Omega))$, together with the bounds \fref{bound4} and Lemma \ref{lemL2}, imply that $\hat \varrho_N$ converges to $\varrho_{eq}[n]$. Since $\hat \varrho_N$ converges as well to $\varrho$ weakly-$*$ in $L^\infty(0,T,\calJ_1)$, this shows that $\varrho=\varrho_{eq}[n]$. \\

It remains to identify $A_{eq}$ with $A$, which is done with the representation formula \fref{repform}.

\paragraph{Step 2: Passing to the limit in the representation formula.} According to \fref{repform}, we know that $A_{eq}$ reads, since we have just proved that $\varrho=\varrho_{eq}[n]$,
\be \label{repform2}
A_{eq}=-V_0+\frac{1}{n}\left(\frac{1}{2}\Delta n +n[\nabla \varrho \nabla]-n[\varrho \log \varrho] \right).
\ee
We want to recover the right-hand side above by passing to the limit in the representation formula for $\hat A_N$. We have then, for any $\varphi \in C^2([0,T] \times \overline{\Omega})$, periodic in $x$,
\bee
\int_0^T (\hat n_N \hat A_N, \varphi) dt&=& \int_0^T \big(-\hat n_N V_0 +\frac{1}{2}\Delta \hat n_N  +n[\nabla \hat{\varrho}_N \nabla]-n[\hat \varrho_N \log \hat \varrho_N],\varphi\big) dt\\
&=&-\int_0^T \big(\hat n_N V_0,\varphi\big) dt +\frac{1}{2}\int_0^T (\hat n_N,\Delta \varphi) dt\\
&& -\int_0^T  \Tr \big( \nabla \hat{\varrho}_N \nabla \, \varphi \big) dt-\int_0^T \Tr \big (\hat \varrho_N \log \hat \varrho_N \, \varphi\big) dt.
\eee
Owing to the weak-$*$ convergence of $\hat n_N$ in $L^\infty(0,T,\Hp)$, passing to the limit in the first two terms in the r.h.s. presents no difficulty and yields the sum
$$
-\int_0^T \big(n V_0,\varphi\big) dt +\frac{1}{2}\int_0^T (n,\Delta \varphi) dt.
$$
Besides, the strong convergence of $\hat n_N$ in $L^2(0,T,L^2(\Omega))$ combined with the weak convergence of $\hat A_N$ in $L^2(0,T,\Hp)$ show that the l.h.s. converges to 
$$
\int_0^T ( n A, \varphi) dt.
$$
The two remaining terms are treated as follows: Write, using the cyclicity of the trace in the third line,
\begin{align*}
&\int_0^T  \Tr \big( \nabla \hat{\varrho}_N \nabla \, \varphi \big) dt \\
&=
\int_0^T  \Tr \big( \nabla (H_0 + \II)^{-1} (H_0 + \II) \hat{\varrho}_N (H_0 + \II)(H_0 + \II)^{-1} \nabla \, \varphi \big) dt\\
&=\int_0^T  \Tr \big( (H_0 + \II) \hat{\varrho}_N (H_0 + \II)(H_0 + \II)^{-1} \nabla \, \varphi \nabla (H_0 + \II)^{-1} \big) dt\\
&:=\int_0^T  \Tr \big( (H_0 + \II)\hat{\varrho}_N (H_0 + \II)\, K\big) dt,
\end{align*}
where $K$ is a compact operator on $L^2(\Omega)$. From the fact that $H_0 \hat \varrho_N H_0 \to H_0 \hat \varrho H_0$ weakly-$*$ in $L^2(0,T,\calJ_1)$, we can conclude that $(H_0 + \II) \hat{\varrho}_N (H_0 +\II) \to (H_0 + \II)\varrho (H_0 + \II)$ weakly-$*$ in $L^2(0,T,\calJ_1)$, and therefore that 
$$
\lim_{N \to \infty}\int_0^T  \Tr \big( \nabla \hat{\varrho}_N \nabla \, \varphi \big) dt = \int_0^T  \Tr \big( \nabla \varrho \nabla \, \varphi \big) dt.
$$
Regarding the last term involving $\hat \varrho_N \log \hat \varrho_N$, we have the following lemma, proved at the end of the section:
\begin{lemma} \label{lemweak} For almost all $t$ in $(0,T)$, the operator $\beta(\hat \varrho_N(t))$ converges weakly in $\calJ_1$ to $\beta( \varrho(t))$.
\end{lemma}

We then write (mostly for notational convenience),
$$
\int_0^T \Tr \big (\hat \varrho_N \log \hat \varrho_N \, \varphi\big) dt=\int_0^T \Tr \big (\beta(\varrho_N) \, \varphi\big) dt+\int_0^T (\hat n_N, \varphi) dt.
$$
Then, according to estimates \fref{ent2} and \fref{rhoE}, we have
$$
|\Tr \big (\beta(\varrho_N) \, \varphi\big)|\leq \|\varphi\|_{L^\infty(0,T,L^\infty)} \| \beta(\hat \varrho_N)\|_{L^\infty(0,T,\calJ_1)}\leq C,
$$
which, using dominated convergence together with Lemma \ref{lemweak} leads to
$$
\lim_{N \to \infty} \int_0^T \Tr \big (\beta(\hat \varrho_N)\, \varphi\big) dt=\int_0^T \Tr \big ( \beta(\varrho)\big) dt.
$$
Finally, the latter, together the weak-$*$ convergence of $\hat n_N$ to $n$ in $L^\infty(0,T,\Hp)$, shows that 
$$ 
\lim_{N \to \infty} \int_0^T \Tr \big (\hat \varrho_N \log \hat \varrho_N \, \varphi\big) dt= \int_0^T \Tr \big ( \varrho \log  \varrho \, \varphi\big) dt.
$$
Collecting the expression of $A_{eq}$ given in \fref{repform2}, and the various limits that we obtained, we can conclude that $A_{eq}=A$. Hence, $n$ and $A$ satisfy the closure relation \fref{clos}. In order to conclude the proof of existence, it remains to prove Lemma \ref{lemL2} and Lemma \ref{lemweak}.

\paragraph{Proof of Lemma \ref{lemL2}.} We show first by duality that $\Delta \hat n_N \in L^2(0,T,L^2(\Omega))$. Indeed, for a test function $\varphi$, we have
$$
\int_0^T \big((\Delta-\II) \hat n_N, \varphi \big) dt=\int_0^T \Tr \big(\hat \varrho_N  (\Delta \varphi-\varphi)  \big) dt.
$$
The latter can be controlled by
$$
\|(\Delta -\II) \hat \varrho_N (\Delta-\II)\|_{L^2(0,T,\calJ_1)} \|  (\Delta-\II)^{-1} (\Delta\varphi -\varphi) (\Delta-\II)^{-1}\|_{L^2(0,T,\calL(L^2(\Omega)))},
$$
which, according to \fref{bound1},\fref{bound6} and Sobolev embeddings, can be bounded by $C\|\varphi\|_{L^2(0,T,L^2)}$. This shows that the limit $n$ of $\hat n_N$ is such that $\Delta n \in L^2(0,T,L^2(\Omega))$. We use now the representation formula \fref{repform}: $A_{eq}$ reads
$$
A_{eq}=-V_0+\frac{1}{n}\left(\frac{1}{2}\Delta n +n[\nabla \varrho_{eq}[n] \nabla]-n[\varrho_{eq}[n] \log \varrho_{eq}[n]] \right).
$$
A direct adaption of \cite{MP-JSP}, Theorem 3, shows that $A_{eq}$ satisfies the estimate, $t$ a.e.,
\bea \label{AL2jsp}
\|A_{eq}(t)\|_{L^2(\Omega)} &\leq& \frac{C}{\underline{n}} \left(\mathfrak{H}_0(n(t))\left(1+\frac{1}{\underline{n}} \left(\| \Delta n(t) \|_{L^2(\Omega)}+\mathfrak{H}_0(n(t))\right)\right)\right)\\
&&+\frac{C}{\underline{n}}\bigg(\exp\left( C (\mathfrak{H}_1(n(t)))^4\right)\bigg)+\|V_0\|_{L^2(\Omega)}, \nonumber
\eea
where $n$ is such that $n \geq \underline{n}$ a.e. on $(0,T)\times \Omega$ and 
\begin{align*}
&\mathfrak{H}_0(n)=1+\beta(\|n\|_{L^1(\Omega)})+\| \sqrt{n}\|^2_{H^1(\Omega)}\\
&\mathfrak{H}_1(n)=\left(1+\|\sqrt{n}\|_{H^1(\Omega)}/\sqrt{\underline{n}} \right)\mathfrak{H}_0(n)/\underline{n}.
\end{align*}
Since $n \in L^\infty(0,T,\Hp)$ and $n \geq \underline{n}$, we have that $\mathfrak{H}_0(n(t))$ and $\mathfrak{H}_1(n(t))$ are bounded by a constant independently of $t$. Since moreover $V_0 \in L^\infty(\Omega)$  and $\Delta n \in L^2(0,T,L^2(\Omega))$, we deduce from \fref{AL2jsp} that $A_{eq}\in L^2(0,T,L^2(\Omega))$. This ends the proof.

\paragraph{Proof of Lemma \ref{lemweak}.} We define for $s \geq 0$ and some $\eps>0$,
$$
\beta(s)=\beta_1(s)+\beta_2(s):= \un_{s \leq \eps}\beta(s)+\un_{s > \eps}\beta(s)
$$
and split $n[\beta(\varrho_N)]$ accordingly into $n[\beta_1(\varrho_N)]+n[\beta_2(\varrho_N)]$. Let $M=\sup_{N}\|\hat \varrho_N\|_{L^\infty(0,T,\calL(L^2))}$. Then, there exists a constant $C_M>0$ such that
$$\forall s\in [0,M],\quad \left|s\log s-s\right|\leq C_Ms^{3/4}.$$
Thus, for all $\eps>0$ and $(\rho_p^N)_{p \in \Nm}$ the (nonincreasing) eigenvalues of $\hat \varrho_N$, we have  $t$ a.e.,
\bea \nonumber
\Tr\big(|\beta_1(\hat \varrho_N(t))|\big)&=&\sum_{\rho_p^N(t)\leq \eps}\left|\beta(\rho_p^N(t))\right|\\[3mm] \nonumber
&\leq& C_M\sum_{\rho_p^N(t)\leq \eps}(\rho_p(t)^N)^{3/4}\leq C_M\eps^{1/4}\sum_{\rho_p^N(t)\leq \eps}(\rho_p^N(t))^{1/2}\\ \nonumber
&\leq& C_M\eps^{1/4}\left(\sum_{p\geq 1}p^2\rho_p^N(t)\right)^{1/2}\left(\sum_{p\geq 1}\frac{1}{p^2}\right)^{1/2}\\[3mm]
&\leq &C\eps^{1/4}\left(\Tr \sqrt{H_0}\hat \varrho_N(t) \sqrt{H_0}\right)^{1/2}\leq C\eps^{1/4}, \label{lemweak1}
\eea
where $C$ is independent of $N$ and $t$, and where we used Lemma \ref{lieb2} with $\calH=H_0$ and estimate \fref{bound7}. We treat now the term
$$
\Tr \big (\beta_2(\hat \varrho_N(t)) \, B\big):=f_N(t),
$$
where $B$ is a bounded operator. To this aim, denote $P(t)=\max \left\{p:\,\rho_p(t) > \eps\right\}$, where $(\rho_p)_{p \in \Nm}$ is the nonincreasing sequence of eigenvalues of $\varrho$. Recall that as a consequence of \fref{strongL2}, $\hat \varrho_N \to \varrho$, strongly in $L^2(0,T,\calJ_2)$, and therefore that there is a subsequence such that $\hat \varrho_N(t) \to \varrho(t)$, $t$ a.e. in $\calJ_2$. Then, according to Lemma \ref{lipspec}, we have
\be \label{cvvalprop}
\forall p\in \Nm, \qquad \rho_p^N(t) \to \rho_p(t), \qquad t\; a.e.,
\ee 
and we can choose $N$ sufficiently large so that, $t$ a.e., 
$$\rho_p^N(t)>\eps\mbox{ for all } p\leq P(t)\mbox{ and }\rho_p^N(t)<\eps \mbox{ for all } p> P(t).$$
Besides, following again Lemma \ref{lipspec}, we can choose some eigenbasis $(\phi_p^N)_{p\in\Nm}$ and $(\phi_p)_{p\in\Nm}$ of $\hat \varrho_N$ and $\varrho$, respectively, such that, 
\be
\label{cvvectprop}
\forall p\in \Nm,\qquad \lim_{N \to \infty}\|\phi_p^N(t)-\phi_p(t)\|_{L^2} =0, \qquad t\; a.e..
\ee
Finally, the function $f_N(t)$ reads
$$
f_N(t)=\sum_{p=0}^{P(t)} \beta(\rho_p^N(t)) (\phi_p^N(t), B \phi_p^N(t)).
$$
Then, according to \fref{cvvalprop}-\fref{cvvectprop}, it follows that
$$
\lim_{N \to \infty} f_N(t)= \sum_{p=0}^{P(t)} \beta(\rho_p(t)) (\phi_p(t), B \phi_p(t)), \qquad t\; a.e..
$$
Together with estimate \fref{lemweak1}, this concludes the proof of Lemma \ref{lemweak} and the proof of existence.

\subsubsection{The free energy derivative.} \label{proofderiv} We prove here estimate \fref{derivfree}, which is a major ingredient in the exponential convergence to the equilibrium. Consider first a solution $(n,A,V)$ to the QDD system, and define $\varrho:=\exp(-(H+A))$, which belongs to $L^\infty(0,T,\calE_+)$ according to Theorem \ref{mainth}. Together with $V \in L^\infty(0,T,H^1_0(\Omega))$ and \fref{ent2}, this implies that $F(\varrho) \in L^\infty(0,T)$. Since moreover $A \in L^2(0,T,L^2(\Omega))$, it follows from \fref{freeA} that $F(\varrho)=\calF[n]$. A formal proof of \fref{derivfree} is then direct: write
$$
\frac{d \calF[n(t)]}{dt}=-\int_\Omega \partial_t n (A+1) dx-\int_\Omega  n \partial_t A dx. 
$$ 
It is shown in \cite{QDD-SIAM} that
$$
0=\frac{d}{dt} \|n_0\|_{L^1}=\frac{d}{dt} \|n(t)\|_{L^1}=\frac{d}{dt} \Tr (e^{H+A(t)})=(n,\partial_t A),
$$
and \fref{derivfree} follows by replacing $\partial_t n$ by its expression given in \fref{qdd}. The rigorous justification requires more work since we have a priori no information about $\partial_t A$, and a regularization does not seem straightforward. We then use crucially here the convexity of the free energy $F$ to justify the calculations. First, the G\^ateaux derivative of $F$ at $\varrho$ exists in any direction $u \in \calJ_1$. Indeed, a direct calculation shows that
$$
DF[\varrho](u)=\Tr \big( (\log \varrho +H_0+V_0+V) u\big)=\Tr \big( (V-A) u\big)=(V-A,n[u]),
$$
which is finite whenever $u \in \calJ_1$ since $(A-V)(t) \in L^\infty(\Omega)$ almost everywhere in $t$. By convexity of $F$, we have then, for $t$ a.e. in $(0,T)$, for $h$ sufficiently small that $t+h \in (0,T)$,
\bea \nonumber
F(\varrho(t+h))-F(\varrho(t)) &\geq& DF[\varrho(t)]  (\varrho(t+h)-\varrho(t))\\
&=&(V(t)-A(t),n(t+h)-n(t)).\label{conve11}
\eea 
In the same way, for $t-h \in (0,T)$,
\bea \nonumber
F(\varrho(t))-F(\varrho(t-h)) &\leq& DF[\varrho(t)](\varrho(t)-\varrho(t-h))\\
&=&(V(t)-A(t),n(t)-n(t-h)).  \label{conve}
\eea
Integrating \fref{conve} between $h$ and $s \in (0,T)$, we find
\begin{align*}
&h^{-1} \int_{s-h}^s F(\varrho(t))dt-h^{-1} \int_{0} ^h F(\varrho(t))dt=h^{-1} \int_h^s (F(\varrho(t))-F(\varrho(t-h)))dt \hfill \\
&\hspace{5cm}\leq h^{-1} \int_h^s \int_{t-h}^t \langle \partial_\tau n(\tau),V(t)-A(t)\rangle_{\Hmp,\Hp} d\tau dt.
\end{align*}
Since $F(\varrho) \in L^\infty(0,T)$, the Lebesgue differentiation theorem yields, $s$ a.e. in $(0,T)$,
$$
\lim_{h \to 0 } \left(h^{-1} \int_{s-h}^s F(\varrho(t))dt-h^{-1} \int_{0} ^h F(\varrho(t))dt \right)=F(\varrho(s))-F(\varrho(0)).
$$
On the other hand, since $\partial_t n \in L^2(0,T,\Hmp)$, and $(V-A) \in L^2(0,T,\Hp)$ according to Theorem \ref{mainth}, invoking again the Lebesgue differentiation theorem shows that, $t$ a.e.,
$$
\lim_{h \to 0 } h^{-1} \int_{t-h}^t \langle \partial_\tau n(\tau), V(t)-A(t) \rangle_{\Hmp,\Hp} d\tau =\langle \partial_t n(t), V(t)-A(t)\rangle_{\Hmp,\Hp}.
$$
Dominated convergence then allows us to conclude that (we use here the maximal function of $\langle \partial_t n(t), V(t)-A(t)\rangle_{\Hmp,\Hp}$ as dominating function),
$$
\lim_{h \to 0 }  \int_h^s h^{-1} \int_{t-h}^t \langle \partial_t n(\tau),V(t)-A(t)\rangle_{\Hmp,\Hp} d\tau dt = \int_0 ^s\langle \partial_t n(t), V(t)-A(t)\rangle_{\Hmp,\Hp} dt,
$$
and therefore
$$
F(\varrho(s))-F(\varrho(0)) \leq \int_0 ^s\langle \partial_t n(t), V(t)-A(t)\rangle_{\Hmp,\Hp} dt.
$$
Proceeding as above, the other convexity inequality \fref{conve11} shows that the above inequality is in fact in equality. This means in particular that $F(\varrho(s))$ is absolutely continuous, and that the almost everywhere defined derivative satisfies \fref{derivfree}, after replacing $\partial_t n$ by its expression in the QDD equation \fref{qdd}. This ends the proof. 

\subsection{Exponential convergence to the equilibrium} \label{expoconv}
The main ingredients are the expression of the time derivative of the free energy given in \fref{derivfree}, together with some logarithmic-Sobolev type  inequality derived from \fref{rel}. The first step is to rewrite appropriately \fref{rel} and to specialize it to our problem.

For $W_i \in L^2(\Omega)$, $i=1,2$, $\varrho_i=\exp(-(H_0+W_i))$, $n_i\equiv n[\varrho_i]$ and $V_i \equiv V[n_i]$ the Poisson potential, we rewrite estimate \fref{rel} as
\bee
S(\varrho_1, \varrho_2)+S(\varrho_2,\varrho_1)&=&(W_2-W_1,n_1-n_2)\\
&=&(W_2-V_2-W_1+V_1,n_1-n_2)-(V_1-V_2,n_1-n_2)\\
&=&(W_2-V_2-W_1+V_1,n_1-n_2)-\| \nabla (V_1-V_2)\|^2_{L^2}.
\eee
Above, we used the Poisson equation to obtain the last term. Denote by $(n_\infty, A_\infty,V_\infty)$ the solution to the stationary problem of Theorem \ref{SP} with constraint $\Tr(\varrho_\infty)=\|n_0\|_{L^1}$, where $n_0$ is the initial condition. Introduce similarly a solution $(n,A,V)$ to the QDD system. With $W_1=V_0+A$ and $W_2=V_0+A_\infty$, $n_1=n$, $n_2=n_\infty$, $\varrho_1=\varrho$, $\varrho_2=\varrho_\infty$, $V_1=V$, $V_2=V_\infty$, using the facts that $A_\infty-V_\infty$ is equal to the constant $-\epsilon_F$ and that $(1,n-n_\infty)=0$, we find
\bee
S(\varrho, \varrho_\infty)+S(\varrho_\infty,\varrho)
&=&(-\epsilon_F-A+V,n-n_\infty)-\| \nabla (V-V_\infty)\|^2_{L^2}\\
&=&-(A-V-(A-V)_\Omega,n-n_\infty)-\| \nabla (V-V_\infty)\|^2_{L^2},
\eee
where we recall that $(A-V)_\Omega$ is the average of $A-V$ over $\Omega$. 
Since $S(\varrho_\infty,\varrho) \geq 0$, the latter equality implies, together with the inclusion $H^1(\Omega) \subset L^\infty(\Omega)$ and the Poincar\'e-Wirtinger inequality, that
\bea \nonumber
S(\varrho, \varrho_\infty) +\frac{1}{2}\| \nabla (V-V_\infty)\|^2_{L^2}&\leq& \| A-V- (A-V)_{\Omega}\|_{L^\infty(\Omega)}\|n-n_\infty\|_{L^1(\Omega)}\\\nonumber
&\leq & C \| A-V- (A-V)_{\Omega}\|_{H^1(\Omega)}\|n-n_\infty\|_{L^1(\Omega)}\\
&\leq& C \| \nabla (A-V)\|_{L^2(\Omega)}\|n-n_\infty\|_{L^1(\Omega)}. \label{estS}
\eea

The second step of the proof is to relate the l.h.s. of the above inequality to the free energy, and the r.h.s. to the dissipation rate of the free energy appearing in \fref{derivfree}. The first part follows from the straightforward lemma below, proved at the end of the section.

\begin{lemma} \label{FE} The free energy satisfies, $t$ a.e.,
$$
F(\varrho(t))-F(\varrho_\infty)=S(\varrho(t),\varrho_\infty)+\frac{1}{2}\|\nabla (V(t)-V_\infty)\|^2_{L^2}.
$$
\end{lemma} 

As a consequence, the l.h.s. of \fref{estS} is simply the difference of the free energies $F(\varrho(t))-F(\varrho_\infty)$. It remains now to relate the r.h.s.. The next key lemma, proved further and based on \fref{estS} and the Klein inequality of Lemma \fref{pinsker}, allows us to control $n-n_\infty$ in $L^1(\Omega)$ in terms of $\nabla (A-V)$ in $L^2(\Omega)$.
\begin{lemma} \label{LL1} The following estimate holds:
$$
\| n-n_\infty \|_{L^1(\Omega)} \leq C \| \nabla (A-V)\|_{L^2(\Omega)}.
$$
\end{lemma}
At this point, we have therefore obtained the inequality
$$
F(\varrho(t))-F(\varrho_\infty) \leq C \| \nabla (A-V)(t)\|^2_{L^2(\Omega)}.
$$
Using the fact that $n \geq \underline{n}$, a.e. on $(0,T) \times \Omega$, we can exhibit the free energy dissipation rate in r.h.s. in order to obtain
\be \label{dissip}
F(\varrho(t))-F(\varrho_\infty) \leq C \underline{n}^{-1} \| \sqrt{n}(t) \nabla (A-V)(t)\|^2_{L^2(\Omega)}.
\ee
As mentioned in the introduction, this inequality can be seen as non-commutative log-Sobolev inequality for the operator $\varrho$. We have everything needed now to conclude: according to \fref{derivfree}, $t$ almost everywhere,
$$
\frac{d}{dt} F(\varrho(t))=\frac{d}{dt} \left(F(\varrho(t))-F(\varrho_\infty)\right)=-\int_\Omega n(t) |\nabla (A(t)-V(t))|^2 dx.
$$
This, together with  \fref{dissip} leads to
$$
\frac{d}{dt} \left(F(\varrho(t))-F(\varrho_\infty)\right)+C\left( F(\varrho(t))-F(\varrho_\infty) \right) \leq 0,
$$
and the conclusion follows from the Gronwall lemma. It remains to prove Lemmas \ref{FE} and \ref{LL1}

\paragraph{Proof of Lemma \ref{LL1}.} The first step is to obtain the estimate below:
\be \label{LLL1}
\| n-n_\infty \|_{L^1(\Omega)} \leq C \| (\varrho-\varrho_\infty)(1+ |H+A_\infty|)^{1/2}\|_{\calJ_2}.
\ee
We proceed as usual by duality. Let $u=\varrho-\varrho_\infty$ and $R=1+ |H+A_\infty|$. Then, for any smooth function $\varphi$,
 \bee
 \left|(n[u],\varphi)\right|&=&\left|  \Tr \big( u \varphi\big) \right|=\left|  \Tr \big( u R^{1/2} R^{-1/2} \varphi\big) \right|\\
 &\leq& \| u R^{1/2}\|_{\calJ_2}\| R^{-1/2} \varphi\|_{\calJ_2}.
 \eee
For $(\lambda_p,u_p)_{p \in \Nm}$ the spectral elements of $R$, and  $(e_k)_{k \in \Nm}$ any basis of $L^2(\Omega)$, the last term satisfies
\bee
\| R^{-1/2} \varphi\|^2_{\calJ_2}&=&\sum_{k \in \Nm } \left\| R^{-1/2} \varphi e_k \right\|^2_{L^2(\Omega)}=\sum_{k \in \Nm }\sum_{p \in \Nm } \left|(R^{-1/2} \varphi e_k,u_p)\right|^2\\
&=&\sum_{k \in \Nm }\sum_{p \in \Nm } \lambda^{-1}_p \left|(\varphi e_k,u_p)\right|^2 =\sum_{p \in \Nm } \lambda_p^{-1}\left\|\overline{\varphi} u_p \right\|^2_{L^2(\Omega)} \\
&\leq& \|\varphi\|_{L^\infty(\Omega)} \sum_{p \in \Nm } \lambda^{-1}_p.
\eee
Now, since $A_\infty=V_\infty-\epsilon_F$ and $V_\infty \in H^1_0(\Omega)$, the minimax principle shows that the eigenvalues of $H+A_\infty$, indexed by $p \in \Nm$, are bounded below by $C p^2 +C'$, with $C>0$, which is positive for $p$ sufficiently large. As consequence $0<1+Cp^2+C'\leq \lambda_p $ for large $p$, and the sum above is finite. This proves estimate \fref{LLL1}.

The second step is to control the r.h.s. of \fref{LLL1} by the relative entropy between $\varrho$ and $\varrho_\infty$ with the goal of using \fref{estS}. We use for this a slightly different version of the Klein inequality of Lemma \ref{pinsker}. We claim that \be \label{entbelow}
C \Tr \big( (1+ |H+A_\infty|)^{1/2}(\varrho-\varrho_\infty)^2(1+ |H+A_\infty|)^{1/2}\big) \leq S(\varrho,\varrho_\infty).
\ee
Together with \fref{LLL1} and \fref{estS}, this ends the proof of lemma provided we justify \fref{entbelow}, which is only a matter of properly using the cyclicity of the trace. Take two operators $\varrho_1$ and $\varrho_2$ in $\calE_+$,  with $S(\varrho_1,\varrho_2)<\infty$. Let $P_k$ be the spectral projector on the first $k$ modes of $\varrho_2$. According to Lemma \ref{pinsker}, we have
\be \label{defsi}
\Tr \big( (1+ |\log(P_k\varrho_2 P_k )|)(P_k \varrho_1 P_k -P_k \varrho_2 P_k )^2\big) \leq S(P_k \varrho_1 P_k,P_k \varrho_2 P_k).
\ee
Since $(1+ |\log(P_k\varrho_2 P_k )|)^{1/2}$ is a bounded operator (indeed the eigenvalues of $\varrho_2$ are positive, nonincreasing, and converging to zero), cyclicity of the trace shows that the l.h.s. of the above inequality reads
\be \label{defsi2}
\Tr \big( \sigma_k \big):=\Tr \big( (1+ |\log(P_k\varrho_2 P_k )|)^{1/2}(P_k \varrho_1 P_k -P_k \varrho_2 P_k )^2(1+ |\log(P_k\varrho_2 P_k )|)^{1/2}\big).
\ee
It just remains to pass to the limit. According to \cite{LS}, Theorem 2, we have, since $P_k \to \II$ strongly in $\calL(L^2(\Omega))$,
\be \label{convS}
 \lim_{k \to \infty} S(P_k \varrho_1 P_k,P_k \varrho_2 P_k)=S(\varrho_1 , \varrho_2 ).
\ee
On the other hand, we conclude from \fref{defsi}, \fref{defsi2} and \fref{convS}, that there is a $\sigma \in \calJ_1$, nonnegative, such that
$$
\sigma_k \to \sigma \quad \textrm{weak-$*$ in }\calJ_1 \quad \textrm{and} \quad \Tr \big(\sigma \big) \leq \liminf_{k \to \infty } \Tr \big( \sigma_k \big) \leq S(\varrho_1 , \varrho_2 ).
$$
Since finally $P_k \varrho_i P_k \to \varrho_i$, strongly in $\calJ_1$ for $i=1,2$, we can identify $\sigma$ with $(1+ |\log(\varrho_2 )|)^{1/2}( \varrho_1 - \varrho_2 )^2(1+ |\log(\varrho_2 )|)^{1/2}$. Indeed we have, for all compact operator $K$, 
$$
\lim_{k \to \infty} \Tr \big( \sigma_k K\big)=\Tr\big( \sigma K\big).
$$
Choosing for instance $K= (1+ |\log(\varrho_2)|)^{-1/2} K_0 (1+ |\log(\varrho_2)|)^{-1/2}$ for $K_0$ compact then yields the result.

\paragraph{Proof of Lemma \ref{FE}.} The proof is a simple calculation. Since $A \in L^2(0,T,L^2(\Omega))$ and $A_\infty \in L^2(\Omega)$, we can use relation \fref{freeA} of Lemma \ref{lemlip} to arrive at
\bee
F(\varrho(t))-F(\varrho_\infty)&=&-\int_\Omega n(t)(A(t)+1) dt+\int_\Omega n_\infty(A_\infty+1)dt\\
&&+\frac{1}{2}\|\nabla V(t)\|^2_{L^2}-\frac{1}{2}\|\nabla V_\infty\|^2_{L^2}\\
&=&-\int_\Omega \left(n(t)(A(t)-A_\infty) +n(t)-n_\infty \right)dt\\
&&+\int_\Omega n_\infty A_\infty dt-\int_\Omega n A_\infty dt+\frac{1}{2}\|\nabla V(t)\|^2_{L^2}-\frac{1}{2}\|\nabla V_\infty\|^2_{L^2}.
\eee
Using the facts that $A_\infty-V_\infty=\epsilon_F$ is constant, that $\|n(t)\|_{L^1}=\|n_\infty\|_{L^1}=\|n_0\|_{L^1}$, and that $-\Delta (V-V_\infty)=n-n_\infty$ equipped with Dirichlet boundary conditions, we obtain the desired result. This ends the proof of the lemma and of the convergence to the equilibrium.


\section{Appendix}
\subsection{Proof of Lemma \ref{regmin}}
We work with a regular periodic potential $V \in C^\infty(\overline{\Omega})$ and obtain the final result by density. The Hamiltonian $H_0+V$ with domain $H^2_{per}$ has a compact resolvant, and we denote by $(\mu_p,\phi_p)_{p \in \Nm}$ its spectral decomposition, with the sequence $(\mu_p)_{p \in \Nm}$ nondecreasing. The minimax principle shows that
\be \label{controlmu}
\frac{1}{2} \gamma_p-C \|V\|^2_{L^2}-C \leq \mu_p \leq \frac{3}{2} \gamma_p +C \|V\|^2_{L^2}+C,
\ee
where $\gamma_p=(2 \pi p)^2$ is an eigenvalue of $H_0$. We have moreover the direct estimate
$$
\| \nabla \phi_p\|^2_{L^2} \leq  C|\mu_p|+C\|V\|^2_{L^2} +C.
$$
This yields, for $\varrho=\exp(-(H_0+V))$,
$$
\Tr \big( \sqrt{H_0} \varrho \sqrt{H_0}\big)=\sum_{p \in \Nm} e^{- \mu_p} \|\nabla \phi_p\|^2_{L^2} \leq C\sum_{p \in \Nm} e^{- C \gamma_p+C} (1+\gamma_p) < \infty,
$$
and therefore $\varrho \in \calE_+$. We turn now to estimate \fref{controlmu}. There are several ways to control $H_0 \varrho H_0$ in $\calJ_1$, and since the system \fref{semiqdd}-\fref{poisson2}-\fref{den2} provides us with direct bounds for $\varrho$ in $\calE_+$ and for the chemical potential $A$ in $L^2$, we estimate $H_0 \varrho H_0$ in terms of these quantities. We write first
\bee
\Tr\big( H_0 \varrho H_0 \big)&=&\Tr\big( (H_0+V) \varrho (H_0+V) \big)-\Tr\big( (H_0+V) \varrho V \big)\\
&&-\Tr\big( V \varrho (H_0+V) \big) +\Tr\big( V \varrho V \big)\\
&:=&T_1+T_2+T_3+T_4.
\eee
All terms above are well defined since $V$ is bounded and $H_0+V$ is bounded below, so that $(H_0+V) \varrho (H_0+V)$ is trace class. We start with the term $T_1$. 

\paragraph{The term $T_1$.} Let $N(V) \in \Nm$ such that $\mu_p \leq 0$ for $p \leq N(V)$, and $\mu_p>0$ for $p>N(V)$. Note that $N(V)$ is finite since $H_0+V$ is bounded below. Using the fact that $\forall \eps \in (0,1)$, there exists $C_\eps>0$ such that,  $ \forall x \geq 0$, $x^2 e^{-x} \leq C_\eps (e^{-x})^{1-\eps}$, we obtain that
\bea \nonumber
T_1&=&\sum_{p \in \Nm} \mu_p^2 e^{-\mu_p} \leq \mu_0^2 \sum_{p \leq N(V)}  e^{-\mu_p}+\sum_{p > N(V)} \mu_p^2 e^{-\mu_p}\\
&\leq & \mu_0^2 \, \Tr \big( \varrho\big)+ C_\eps \Tr \big( \varrho^{1-\eps}\big). \label{eT1}
\eea
In the first term, we control $|\mu_0|$ using the minimax principle:
\bee
0 \geq \mu_0&=& \min_{\phi \in \Hp, \|\phi\|_{L^2}=1} \left( \|\nabla \phi \|^2_{L^2}+ (V,|\phi|^2) \right)\\
&\geq & - \| V \|_{L^2}  \max_{\phi \in \Hp, \|\phi\|_{L^2}=1} \|\phi\|^2_{L^4} \geq - C \| V \|_{L^2},
\eee
which gives
\be \label{mu0}
 |\mu_0| \leq  C\|V\|_{L^2}.
 \ee
For the second term in \fref{eT1}, we denote by $(\rho_p)_{p \in \Nm}$ (with $\rho_p=e^{- \mu_p})$ the eigenvalues of $\varrho$, and by $(\lambda_p)_{p \in \Nm}$ those of $H_0+\II$ (with of course $\lambda_p=\gamma_p+1$). Following Lemma \ref{lieb2} with $\calH=H_0+\II$, we find
\bea \nonumber
\Tr \big(\varrho^{1-\eps}\big)&=&\sum_{p \in \Nm} \rho_p^{1-\eps} \leq \left(\sum_{p \in \Nm}\rho_p \lambda_p \right)^{1-\eps} \left( \sum_{p \in \Nm}  \lambda_p^{-(1-\eps)/\eps}\right)^{\eps}\\[3mm]
&\leq & C \Big( \Tr \big( (H_0+\II)^{1/2} \varrho (H_0+\II)^{1/2} \big) \Big)^{1-\eps}, \label{eT11}
\eea
since $\lambda_p=(2 \pi p)^2+1$ and whenever $\eps<2/3$, we have
$$
\sum_{p \in \Nm}  \lambda_p^{-(1-\eps)/\eps} < \infty.
$$
Remarking that
\be \label{rem}
\qquad \forall u \in \calE_+, \qquad \Tr \big( \sqrt{H_0}u \sqrt{H_0} \big) + \Tr \big( u\big) =\Tr \big( (H_0+\II)^{1/2}u (H_0+\II)^{1/2} \big),
\ee
we find, together with \fref{eT1}, \fref{mu0}, and \fref{eT11}, that  $T_1$ can be estimated as follows, for all $\eps \in (0,2/3)$,
\bea \nonumber
T_1 &\leq& C\|V\|^2_{L^2}\Tr \big( \varrho\big)+C\left( \Tr \big( \sqrt{H_0} \varrho \sqrt{H_0} \big)+ \Tr \big( \varrho \big) \right)^{1-\eps} \\
&\leq &C\big(1+\|V\|^2_{L^2}\big)\Tr \big( \varrho\big)+\Tr \big( \sqrt{H_0} \varrho \sqrt{H_0} \big), \label{eT12}
\eea
where we used the Young inequality. 

\paragraph{The term $T_4$.} We turn now to $T_4$, that verifies
$$
T_4=(n[\varrho],V^2),
$$
and it suffices to bound now $n[\varrho]$ in $L^\infty$. This is done by duality: for any $\varphi \in L^\infty(\Omega)$, $\varphi$ nonnegative,
\bee
(n[\varrho],\varphi)&=&\Tr \big( \varrho \varphi\big)=\Tr \big( (H_0+\II)^{1/2}\varrho (H_0+\II)^{1/2}(H_0+\II)^{-1/2}\varphi (H_0+\II)^{-1/2}\big)\\
&\leq &\Tr \big( (H_0+\II)^{1/2}\varrho (H_0+\II)^{1/2} \big) \|(H_0+\II)^{-1/2}\varphi (H_0+\II)^{-1/2}\|_{\calL(L^2)}\\
&\leq&C\Tr \big( (H_0+\II)^{1/2}\varrho (H_0+\II)^{1/2} \big)\|\varphi\|_{L^1}
\eee
since $(H_0+\II)^{-1/2}$ is bounded from $L^1(\Omega)$ to $L^\infty(\Omega)$. Accounting for \fref{rem}, we obtain the estimate
\be
T_4 \leq C \|V\|^2_{L^2}\left( \Tr \big( \sqrt{H_0}\varrho \sqrt{H_0} \big) + \Tr \big( \varrho\big) \right). \label{eT4}
\ee
\paragraph{The term $T_2$.} We consider now the term $T_2$, that we first control by, proceeding in the same way as for the term $T_1$, 
$$
|T_2|=|(n[(H_0+V) \varrho],V)| \leq |\mu_0| (n[\varrho],|V|)+C_\eps (n[\varrho^{1-\eps}],|V|).
$$
Using \fref{mu0} and the $L^\infty$ estimate for $n[\varrho]$, we find for the first term
\be \label{t41}
 |\mu_0| (n[\varrho],|V|) \leq C \|V\|^2_{L^2}\left( \Tr \big( \sqrt{H_0}\varrho \sqrt{H_0}\big)  +\Tr \big( \varrho \big) \right).
\ee
For the second term, let $\gamma \in (0,1/2)$, and write
\bee
\Tr \big(\varrho^{1-\eps} |V| \big)&=&\Tr \big((H_0+\II)^{\gamma(1-\eps)}\varrho^{1-\eps}(H_0+\II)^{\gamma(1-\eps)}(H_0+\II)^{-\gamma(1-\eps)}|V|(H_0+\II)^{-\gamma(1-\eps)}\big)\\
&\leq& \Tr \big((H_0+\II)^{\gamma(1-\eps)}\varrho^{1-\eps}(H_0+\II)^{\gamma(1-\eps)} \big) \\
&& \hspace{2cm} \times \|(H_0+\II)^{-\gamma(1-\eps)}|V|(H_0+\II)^{-\gamma(1-\eps)}\|_{\calL(L^2)}.
\eee
We estimate the first term in the r.h.s. with the Araki-Lieb-Thirring inequality:
\begin{align*}
&\Tr \big((H_0+\II)^{\gamma(1-\eps)}\varrho^{1-\eps}(H_0+\II)^{\gamma(1-\eps)} \big) \leq  \Tr \left(\big((H_0+\II)^{\gamma}\varrho (H_0+\II)^{\gamma} \big)^{1-\eps}\right).
\end{align*}
The last term is controlled by using Lemma \ref{lieb2}: Denoting by $(\nu_p)_{p\in \Nm}$ the eigenvalues of $(H_0+\II)^{\gamma}\varrho (H_0+\II)^{\gamma}$, we have
\bee
\sum_{p \in \Nm} \nu_p^{1-\eps} &\leq &\left(\sum_{p \in \Nm} \nu_p \lambda_p^{1-2\gamma} \right)^{1-\eps}  \left(\sum_{p \in \Nm} (\lambda_p^{1-2\gamma})^{-(1-\eps)/\eps}\right)^{\eps}\\
&\leq& C\Big(  \Tr \big( (H_0+\II)^{1/2}\varrho (H_0+\II)^{1/2} \big) \Big)^{1-\eps} \left(\sum_{p \in \Nm} (\lambda_p^{1-2\gamma})^{-(1-\eps)/\eps}\right)^{\eps},
\eee
where we used Lemma \ref{lieb2} with $\varrho \equiv (H_0+\II)^{\gamma}\varrho (H_0+\II)^{\gamma}$ and $H\equiv (H_0+\II)^{1-2\gamma}$ in the last line. The sum above is finite whenever $\eps<2(1-2\gamma)/(3-4\gamma)$ since $\lambda_p=(2 \pi p)^2+1$. Besides, we have the inequality
$$
\|(H_0+\II)^{-\gamma(1-\eps)}|V|(H_0+\II)^{-\gamma(1-\eps)}\|_{\calL(L^2)} \leq C \|V\|_{L^2}, \qquad \gamma (1-\eps)>1/4,
$$
since $H^{s}(\Omega) \subset L^\infty(\Omega)$ when $s>1/2$. Setting for instance $\gamma=3/8$, and $\eps \in (0,1/3)$, and using the Young inequality and \fref{rem}, this allows us to control $\varrho^{1-\eps} |V|$ by
$$
\Tr \big(\varrho^{1-\eps} |V| \big) \leq C+C\Tr \big( \sqrt{H_0} \varrho \sqrt{H_0} \big)+C\Tr\big( \varrho \big)+C\|V\|^2_{L^2}.
$$
Together with \fref{t41}, we finally find for $T_2$,
\be \label{eT2}
|T_2| \leq C+ C\|V\|^2_{L^2}+C\Tr \big( \sqrt{H_0} \varrho \sqrt{H_0} \big)+ C \Tr \big( \varrho \big).
\ee
The term $T_3$ is estimated in the same fashion as $T_2$. Collecting \fref{eT12}, \fref{eT4} and \fref{eT2} finally yields the desired result. This ends the proof of the lemma.

{\footnotesize \bibliographystyle{siam}
  \bibliography{bibliography} }

 \end{document}